\documentclass[11pt]{amsart}
\usepackage{fullpage, float}
\usepackage{amssymb,amsmath,amscd,mathrsfs,graphicx, hyperref, enumerate}
\usepackage[cmtip,all,matrix,arrow,tips,curve]{xy}
\usepackage{tikz-cd, tikz, verbatim, pgfplots}
\usetikzlibrary{shapes,snakes}
\usepackage{caption}
\theoremstyle{plain}
\newtheorem{theorem}{Theorem}[section]
\newtheorem{lemma}{Lemma}[section]
\newtheorem{prop}{Proposition}[section]

\theoremstyle{definition}
\newtheorem{definition}{Definition}[section]
\newtheorem{remark}{Remark}[section]

\newcommand{\PP}{\mathbb{P}}
\newcommand{\CC}{\mathbb{C}}
\newcommand{\QQ}{\mathbb{Q}}
\newcommand{\calO}{\mathcal{O}}
\newcommand{\calN}{\mathcal{N}}
\newcommand{\calB}{\mathcal{B}}
\newcommand{\calR}{\mathcal{R}}
\newcommand{\calM}{\mathcal{M}}
\newcommand{\mm}{\mathbb{Z}/2\mathbb{Z}}
\newcommand{\PGL}{\mathbb{P}\mathrm{GL}}
\newcommand{\SL}{\mathrm{SL}}
\newcommand{\SU}{\mathrm{SU}}
\newcommand{\GL}{\mathrm{GL}}
\newcommand{\diag}{\mathrm{diag}}
\newcommand{\Bl}{\mathrm{Bl}}
\newcommand{\rk}{\mathrm{rk}}
\newcommand{\Stab}{\mathrm{Stab}}
\newcommand{\Aut}{\mathrm{Aut}}

\title[Cohomology of the moduli space of non-hyperelliptic genus four curves]{Cohomology of the moduli space of non-hyperelliptic genus four curves}
\author[M. Fortuna]{Mauro Fortuna}
\address{Institut f\"{u}r Algebraische Geometrie, Leibniz Universit\"{a}t Hannover, Welfengarten 1, 30167 Hannover, Germany.}
\email{fortuna@math.uni-hannover.de}
\begin{document}
	
	\begin{abstract}
		We compute the intersection Betti numbers of the GIT model of the moduli space of Brill-Noether-Petri general curves of genus 4. This space was shown to be the final non-trivial log canonical model for the moduli space of stable genus four curves, under the Hassett-Keel program. The strategy of the cohomological computation relies on a general method developed by Kirwan to calculate the cohomology of GIT quotients of projective varieties, based on the equivariantly perfect stratification of the unstable points studied by Hesselink and others and a partial resolution of singularities.
	\end{abstract}
	\maketitle
	
	\section{Introduction}
Moduli spaces of curves and their geometrically meaningful compactifications are a central topic in algebraic geometry. In particular, one wants to understand the topology of these spaces. From that perspective, the purpose of this paper is to compute the intersection Betti numbers of the moduli space of non-hyperelliptic Brill-Noether-Petri-general curves of genus 4. The canonical model of such curves is a complete intersection of a smooth quadric and a cubic surface in projective space. This moduli space hence carries a natural compactification:
\[ M:= \PP H^0(\PP^1\times\PP^1, \calO_{\PP^1\times\PP^1}(3, 3))  /\!\!/ \text{Aut}(\PP^1\times\PP^1), \]
as GIT quotient for the space of curves of bidegree $(3,3)$ on $\PP^1\times \PP^1$ under automorphism.

The variety $M$ is a projective birational compactification for the moduli space of genus $4$ curves $M_4$, which is the coarse moduli space associated to the moduli stack $\calM_4$. The study of the birational models for $M_g$ is the subject of the Hassett-Keel program (see \cite{Has05}), which aims at giving a modular interpretation of the canonical model of the Deligne-Mumford compactification $\overline{M_g}$. The genus $4$ case has attracted a lot of attention as non-trivial instance of the aforementioned program. Specifically, Fedorchuk \cite{Fed12} proved that $M$ is the final non-trivial log canonical model for $ \overline{M_4} $, namely
\[ M\cong \overline{M_4}(\alpha):= \mathrm{Proj}\bigoplus_{n\geq 0} H^0(n(K_{\overline{M_4}}+\alpha \delta)), \quad \alpha \in \left( \frac{8}{17}, \frac{29}{60} \right]\cap \QQ,  \]
where $ \delta \subset \overline{M_4} $ is the boundary divisor. In \cite{CMJL12} and \cite{CMJL14}, Casalaina-Martin, Jensen and Laza described the last steps of the Hassett-Keel program for log minimal models of $ \overline{M_4} $, arising as VGIT quotients of the parameter space of $ (2,3) $ complete intersections. On the other hand, Hassett, Hyeon and Lee (see \cite{HH09}, \cite{HH13} and \cite{HL14}) proved that the program starts with a divisorial contraction, followed by a flip and a small contraction and gave a modular interpretation of the resulting spaces. In conclusion, most of the Hassett-Keel program for genus $4$ is currently known. From our perspective, the salient point is the two ends of program, namely the Deligne-Mumford compactification $\overline{M_g}$ and the GIT quotient for $(3,3)$ curves.

We are interested in examining these spaces from a cohomological point of view. The study of the cohomology of moduli spaces of curves is a subject of great interest in algebraic geometry (see e.g. \cite{Mum83} and \cite{FP15}). For genus 4, one has a complete understanding of the rational cohomology of the Deligne-Mumford compactification $ \overline{M_4} $ due to Bergstr\"{o}m-Tommasi in \cite{BT07}, and that of $ M_4 $ by Tommasi in \cite{Tom05}. The purpose of this paper is to compute the cohomology at the other end of the Hassett-Keel program, namely the Betti numbers of $M$. We want to point out that the cohomology of $M$ and $\overline{M_4}$ are in principle related by the wall crossing along the Hassett-Keel program: we plan to explore this relation in more detail in a future project.

The strategy to compute the intersection Betti numbers of $ M $ relies on a general procedure developed by Kirwan to calculate the cohomology of GIT quotients (see \cite{Kir84}, \cite{Kir85}, \cite{Kir86}). The crucial step of that method consists of the construction of a partial desingularization $ \widetilde{M} \rightarrow M$, known as \textit{Kirwan blow-up}, having only finite quotient singularities, obtained by successively blowing-up the loci parametrising strictly polystable points in the parameter space. Then one is to compute the Hilbert-Poincar\'{e} polynomial of $\widetilde{M}$ and descend back to the GIT quotient $M$ using the \textit{Decomposition Theorem}. 

Examples of application of Kirwan's method are the topological descriptions of the moduli space of points on the projective line (\cite[\S 8]{MFK94}), of $ K3 $ surfaces of degree 2 (\cite{Kir88}) and of hypersurfaces in $ \PP^n $ (\cite{Kir89}), with explicit complete computations only in the case of plane curves up to degree $ 6 $, cubic and quartic surfaces. More recently, the procedure has been applied to compactifications of the moduli space of cubic threefolds (\cite{CMGHL}).

Our result is summarised by the following:

\begin{theorem}\label{thm:main}
	The intersection Betti numbers of $ M $ and the Betti numbers of the Kirwan blow-up $ \widetilde{M} $ are as follows:
	\begin{center}
		\begin{tabular}{r|cccccccccc}
	i&0&2&4&6&8&10&12&14&16&18\\ \hline \rule{0pt}{2.5ex}
	$\dim IH^i(M, \QQ)$&1 &1 &2 &2 &3 &3 &2 &2 &1 &1\\ \rule{0pt}{2.5ex}
	$\dim H^i(\widetilde{M}, \QQ)$& 1 & 4 & 7 & 11 & 14 & 14 & 11 & 7 & 4 &1  
	\end{tabular}
	\end{center}
	while all the odd degree (intersection) Betti numbers vanish.
\end{theorem}

The structure of the paper reflects the steps of Kirwan's machinery. In Section \S \ref{sec:background} we recall the construction of $M$ as GIT quotient $X/\!\!/G$ together with the geometrical description of the semistable and stable loci. In Section \S \ref{sec:strati}, we calculate the equivariant Hilbert-Poincar\'{e} polynomial of the semistable locus $X^{ss}$ in the parameter space of $(3,3)$ curves (see Proposition \ref{prop:seriesss}). This is done by computing the \textit{Hesselink-Kempf-Kirwan-Ness (HKKN) stratification} of the unstable locus, naturally associated to the linear action of $ G $ on the parameter space $ X $, followed by an excision type argument. In Section \S \ref{sec:blow}, we explicitly construct the partial desingularization $ \widetilde{M} \rightarrow M $, by blowing-up three $ G $-invariant loci in the GIT boundary of $ M $, corresponding to strictly polystable curves (see Definition \ref{def:Mtilde}). These subspaces are given by triple conics in $ \PP^1 \times \PP^1 $, curves with two $ D_4 $ or two $ D_8 $ singularities, called \textit{D-curves}, and curves with two singularities of type $ A_5 $, called \textit{A-curves}. Section \S \ref{sec:cohomology} is devoted to the computation of the rational Betti numbers of the Kirwan blow-up $ \widetilde{M} $ (see Theorem \ref{thm:cohoblow}). Here the correction terms arising from the modification process $\widetilde{M} \rightarrow M$ are divided into a main and an extra contribution: the former takes into account the geometry of the centres of the blow-ups and the latter the action of $G$ on the exceptional divisors. In the end, the intersection Betti numbers of $ M $ are computed in Section \S \ref{sec:intersection}, as an application of the \textit{Decomposition Theorem} (cf. \cite{BBD82}) to the blow-down operations at the level of parameter spaces (see Theorem \ref{thm:intM}). We conclude with a geometric interpretation of some Betti numbers, via a description of the classes of curves in the GIT boundary which generate some cohomology groups.

\subsection*{Notation and conventions} We work over the field of complex numbers and all the cohomology and homology theories are taken with \textit{rational} coefficients. The intersection cohomology will be always considered with respect to the middle perversity (see \cite{KW06} for an excellent introduction). For any topological group $ G $, we will denote by $ G^0 $ the connected component of the identity in $ G $ and by $ \pi_0(G):=G/G^0 $ the finite group of connected components of $ G $. The universal classifying bundle of $ G $ will be denoted by $ EG\rightarrow BG $. If $ G $ acts on a topological space $ Y $, its equivariant cohomology (see \cite{AB83}) will be defined to be $ H^*_G(Y):=H^*(Y\times_G EG) $. The Hilbert-Poincar\'{e} series is denoted by 
\[ P_t(Y):=\sum_{i\geq 0}t^i \dim H^i(Y), \]
and analogously for the intersection and equivariant cohomological theories. If $ F $ is a finite group acting on a vector space $ A $, then $ A^F $ will indicate the subspace of elements in $ A $ fixed by $ F $.

\subsection*{Acknowledgements} I wish to thank my advisor Klaus Hulek who proposed me this problem, for many helpful discussions and suggestions. I am also grateful to Yano Casalaina-Martin, Radu Laza and Orsola Tommasi for useful conversations and correspondence and to all the authors of \cite{CMGHL} for kindly sharing a preprint version of it with me. Finally, I would like to thank the anonymous referee whose comments improved the presentation of this paper.

\section{Background on GIT for $(3,3)$ curves in $\PP^1\times \PP^1$}\label{sec:background}

A smooth non-hyperelliptic curve of genus 4 is realised by the canonical embedding as a complete intersection of a quadric and a cubic surface in the projective space $ \PP^3 $. If the quadric is smooth, the curve is said to be \textit{Petri-general} and thus defines a point in the complete linear system
 \begin{equation}
 X:=\PP H^0(\PP^1\times\PP^1, \calO_{\PP^1\times\PP^1}(3, 3))=\PP(\text{Sym}^3(\CC^2)^{\vee}\otimes\text{Sym}^3(\CC^2)^{\vee})\cong\PP^{15}
 \end{equation}
  of curves of bidegree $(3, 3)$ on $\PP^1\times\PP^1$. Since every such curve admits a unique pair of $ g^1_3 $ systems, it follows that these curves are abstractly isomorphic as algebraic curves if and only if they lie in the same $ \Aut(\PP^1 \times \PP^1) $-orbit.
  
  We consider the reductive group $ G:=(\SL(2, \CC)\times \SL(2, \CC))\rtimes \mm $, which is only isogenus to $ \Aut(\PP^1 \times \PP^1)=\PP\mathrm{O(4, \CC)} $, but has the advantage to define a linearisation of the hyperplane bundle of $ X $. We will work with this linearisation throughout all the results. The action of $ G $ on $ X $ is induced by the natural action of $ \SL(2, \CC)\times \SL(2, \CC) $ on $ \PP^1\times\PP^1 $ via change of coordinates and the $ \mm $-extension interchanges the rulings of $ \PP^1\times\PP^1 $. Geometric Invariant Theory \cite{MFK94} provides a good categorical projective quotient 
  \begin{equation}
  M:=X /\!\!/_{\calO_X(1)} G,
  \end{equation}
whose cohomology we aim to compute. In particular, intersection cohomology satisfies Poincar\'{e} duality, allowing us to compute the Betti numbers up to dimension $ 9=\dim M $. However, we prefer to carry out the computations in all dimensions for the sake of completeness, and to report also the results \textit{mod $ t^{10} $} for the sake of readability.

Firstly, we need a description of the semistability condition for non-hyperelliptic Petri-general curves of genus 4. This is provided by the following
\begin{theorem} \cite[2.2]{Fed12}
	A curve $ C $ is unstable (i.e. non-semistable) for the action of $ (\SL(2, \CC)\times \SL(2, \CC))\rtimes \mm $ on $ \PP H^0(\PP^1\times\PP^1, \calO_{\PP^1\times\PP^1}(3, 3)) $ if and only if one of the following holds:
	\begin{enumerate}[(i)]
		\item $ C $ contains a double ruling;
		\item $ C  $ contains a ruling and the residual curve $ C' $ intersects this ruling in a unique point that is also a singular point of $ C' $.
	\end{enumerate}
\end{theorem}

The GIT boundary $ M \smallsetminus M^{s} $ consisting of strictly polystable points is described by the following

\begin{theorem}\label{thm:polystable} \cite[\S 2.2]{Fed12} \cite[3.7]{CMJL14} 
	The strictly polystable curves for the action of $ (\SL(2, \CC)\times \SL(2, \CC))\rtimes \mm $ on $ \PP H^0(\PP^1\times\PP^1, \calO_{\PP^1\times\PP^1}(3, 3)) $ fall into four categories:
	\begin{enumerate}[(i)]
		\item Triple conics;
		\item Unions of a smooth double conic and a conic that is nonsingular along the double conic. These form a one-dimensional family;
		\item Unions of three conics meeting in two $ D_4 $ singularities. These form a two-dimensional family; 
		\item Unions of two lines of the same ruling, meeting the residual curve in two $ A_5 $ singularities.
	\end{enumerate}
\end{theorem}

\section{Equivariant stratification and Hilbert-Poincar\'{e} series}
\label{sec:strati}
In this Section, we discuss the first step of Kirwan's method to compute the cohomology of GIT quotients. It consists of an equivariant stratification of the parameter space measuring the instability of every point under the group action (cf. Theorem \ref{thm:strata}). This stratification turns out to be perfect, in the sense that the Betti numbers of all strata sum up to the cohomology of the whole parameter space (cf. Theorem \ref{thm:equi}). We then apply these results to our case of Brill-Noether-Petri-general curves of genus 4, and we obtain in Proposition \ref{prop:seriesss} the equivariant Betti numbers of the semistable locus.

\subsection{The HKKN stratification}\label{subsec:strata}
From the results in \cite{Kir84}, the first step in Kirwan's procedure is to consider the \textit{Hesselink-Kempf-Kirwan-Ness (HKKN) stratification} of the parameter space, which, from a symplectic viewpoint, coincides to the Morse stratification for the norm-square of an associated moment map. 

In general, let $ X\subset \PP^n $ be a complex projective manifold, acted on by a complex reductive group $ G $, inducing a linearisation on the very ample line bundle $ L=\calO_{\PP^n}(1)|_{X} $. We pick a maximal compact subgroup $ K \subset G $, whose complexification gives $ G $, and a maximal torus $ T\subset G $, such that $ T\cap K $ is a maximal compact torus of $ K $. Before describing the stratification, we need also to fix an inner product together with the associated norm $ \| .\| $ on the dual Lie algebra $ \mathfrak{t}^{\vee}:=\mathrm{Lie}(T\cap K)^{\vee} $, e.g. the Killing form, invariant under the adjoint action of $ K $.
\begin{theorem}\label{thm:strata}\cite{Kir84}
	In the above setting, there exists a natural stratification of $ X$
	\begin{equation}
	X=\bigsqcup_{\beta \in \calB}S_\beta
	\end{equation}
	by $ G $-invariant locally closed subvarieties $ S_{\beta} $, indexed by a finite partially ordered set $ \calB\subset \mathrm{Lie}(T\cap K) $ such that the minimal stratum $ S_0=X^{ss} $ is the semistable locus of the action and the closure of $ S_{\beta} $ is contained in $ \bigcup_{\gamma \geq \beta} S_{\gamma}$, where $ \gamma \geq \beta $ if and only if $ \gamma = \beta $ or $ \|\gamma \| > \|\beta \| $. 
\end{theorem}

We briefly sketch the construction of the strata appearing in the previous Theorem \ref{thm:strata} (see \cite{Kir84} for a detailed description). Let $ \{ \alpha_0,..., \alpha_n \}\subset \mathfrak{t}^{\vee} $ be the weights of the representation (a.k.a. the linearisation) of $ G $ on $ \CC^{n+1} $ and identify $ \mathfrak{t}^{\vee} $ with $ \mathfrak{t} $ via the invariant inner product. After choosing a positive Weyl chamber $ \mathfrak{t}_{+} $, an element $ \beta\in \bar{\mathfrak{t}}_{+} $ belongs to the indexing set $ \calB $ of the stratification if and only if $ \beta $ is the closest point to the origin of the convex hull of some non-empty subset of $ \{ \alpha_0,..., \alpha_n \} $. We define $ Z_{\beta} $ to be the linear section of $ X $
\[ Z_{\beta}:=\{(x_0:...:x_n)\in X: x_i=0 \ \mathrm{if} \ \alpha_i.\beta\neq \|\beta \|^2 \}. \]
The stratum indexed by $ \beta $ is then
\[ S_{\beta}:=G\cdot \bar{Y}_{\beta}\smallsetminus \bigcup_{\|\gamma\|>\|\beta\|} G\cdot \bar{Y}_{\gamma}, \]
where
\[ \bar{Y}_{\beta}:=\{ (x_0:...:x_n)\in X: x_i=0 \ \mathrm{if} \ \alpha_i.\beta< \|\beta \|^2 \}. \]

The heart of Kirwan's results in \cite{Kir84} is the proof that the equivariant Betti numbers of the strata sum up to the cohomology of the whole space.

\begin{theorem}\label{thm:equi} \cite[8.12]{Kir84}
	The stratification $ \{ S_\beta \}_{\beta \in \calB} $, constructed in Theorem \ref{thm:strata}, is $ G $-equivariantly perfect, namely it holds
	\begin{equation}
	P_{t}^{G}(X^{ss})=P_{t}^{G}(X)-\sum_{0\neq\beta \in \calB }t^{2 \mathrm{codim}(S_{\beta})}P_{t}^{G}(S_\beta).
	\end{equation}
\end{theorem}
\begin{remark}\label{rmk:unstable}
	If we denote by $ \Stab \beta \subset G $ the stabiliser of $ \beta \in \mathfrak{t} $ under the adjoint action of $ G $, the equivariant Hilbert-Poincar\'{e} series of each stratum is 
	\[ P_t^G(S_{\beta})=P_t^{\Stab \beta}(Z_{\beta}^{ss}), \]
	where $ Z_{\beta}^{ss} $ is the set of semistable points of $ Z_{\beta} $ with respect to an appropriate linearisation of the action of $ \Stab \beta $ (cf. \cite[8.11]{Kir84}).
\end{remark}
\subsection{Stratification for $ (3, 3) $ curves in $ \PP^1 \times \PP^1 $}
We now come back to our case described in the Section \S \ref{sec:background}. We apply Kirwan's results of the previous subsection to prove the following: 

\begin{prop}\label{prop:seriesss} The $ G$-equivariant Hilbert-Poincar\'{e} series of the semistable locus is 
	\begin{align*}
	P_{t}^{G}(X^{ss})&=\dfrac{1+t^2+t^4+t^6+2t^8+2t^{10}+t^{12}-t^{14}-t^{16}-t^{18}-t^{20}-t^{22}}{1-t^4}\\
	&\equiv P_{t}^{G}(X)\equiv  1+t^2+2t^4+2t^6+4t^8 \ \mathrm{mod} \ t^{10}. 
	\end{align*}
\end{prop}

We need to start computing the equivariant Hilbert-Poincar\'{e} series $ P_{t}^{G}(X) $. Since $ X $ is compact, its equivariant cohomology ring is the invariant part under the action of $ \pi_0(G)=\mm $ of $ H^*_{G^0}(X) $, which splits into the tensor product $ H^*(BG^0)\otimes H^*(X) $ (see \cite[8.12]{Kir84}). Then
\begin{align*}
H^*_G(X)&=H^*_{G^0}(X)^{\mm}\\
&=(H^*(B(\SL(2, \CC)\times \SL(2, \CC)))\otimes H^*(\PP^{15}))^{\mm}\\
&=(\QQ[c_1, c_2]\otimes \QQ[h]/(h^{16}))^{\mm}.
\end{align*}
In fact $ H^*(B\SL(2, \CC))\cong\QQ[c] $, where $ c $ has degree 4, and $ H^*(\PP^n)=\QQ[h]/(h^{n+1}) $, with deg($ h $)=2. The extension $ \mm $ acts by interchanging $ c_1 $ and $ c_2 $, while it fixes the hyperplane class $ h\in H^2(\PP^{15}) $. Therefore the ring of invariants is generated by $ c_1+c_2$, $ c_1c_2 $ and $ h $:
\[ H^*_G(X)=\QQ[c_1+c_2, c_1c_2]\otimes \QQ[h]/(h^{16}). \]
Since deg($ c_1+c_2 $)=4 and deg($ c_1c_2 $)=8, we have:
\begin{align}\label{eq:PGX}
	 P_t^G(X)&=\frac{1+t^2+...+t^{30}}{(1-t^4)(1-t^8)}\\
	 &\equiv 1+t^2+2t^4+2t^6+4t^8 \  \text{mod} \ t^{11}. \nonumber
\end{align}

According to Theorem \ref{thm:equi}, we need to subtract the contributions coming from the unstable strata. In our case, the indexing set $\calB $ of the stratification can be visualised by means of the Figure \ref{fig:1}, called \textit{Hilbert diagram}.

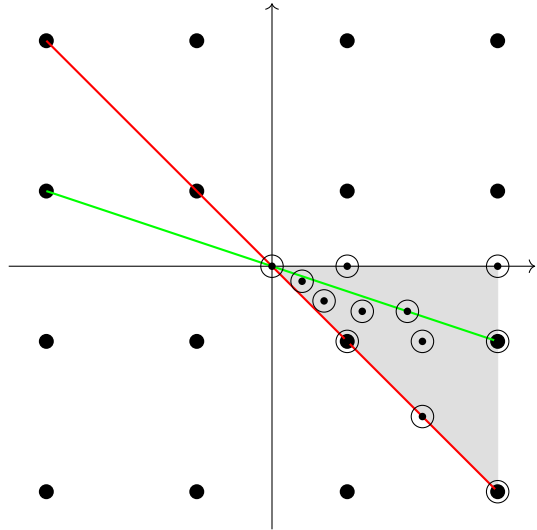
\begin{figure}[htbp] 
	\begin{tikzpicture}
	\draw [lightgray!50, fill=lightgray!50] (0,0) -- (3,0) -- (3,-3) -- (0,0);
	\draw [->] (-3.5,0) -- (3.5,0);
	\draw [->] (0, -3.5) -- (0, 3.5);
	
	\draw (1,1) node[circle,fill, inner sep=2pt]{};
	\draw (1,3) node[circle,fill, inner sep=2pt]{};
	\draw (3,1) node[circle,fill, inner sep=2pt]{};
	\draw (3,3) node[circle,fill, inner sep=2pt]{};
	
	\draw (-1,1) node[circle,fill, inner sep=2pt]{};
	\draw (-1,3) node[circle,fill, inner sep=2pt]{};
	\draw (-3,1) node[circle,fill, inner sep=2pt]{};
	\draw (-3,3) node[circle,fill, inner sep=2pt]{};
	
	\draw (1,-1) node[circle,fill, inner sep=2pt]{};
	\draw (1,-3) node[circle,fill, inner sep=2pt]{};
	\draw (3,-1) node[circle,fill, inner sep=2pt]{};
	\draw (3,-3) node[circle,fill, inner sep=2pt]{};
	
	\draw (-1,-1) node[circle,fill, inner sep=2pt]{};
	\draw (-1,-3) node[circle,fill, inner sep=2pt]{};
	\draw (-3,-1) node[circle,fill, inner sep=2pt]{};
	\draw (-3,-3) node[circle,fill, inner sep=2pt]{};
	
	\draw [black, thick] (-3, 3) -- (3, -3);
	\draw [black, thick] (-3, 1) -- (3, -1);
	
	\draw (0,0)  node[circle,fill,inner sep=1pt]{};
	\draw (0,0)  node[circle,draw, inner sep=3pt]{};
	\draw (3,-3)  node[circle,fill,inner sep=1pt]{};
	\draw (3,-3)  node[circle,draw, inner sep=3pt]{};
	\draw (2, -2)  node[circle,fill,inner sep=1pt]{};
	\draw (2, -2)  node[circle,draw, inner sep=3pt]{};
	\draw (1, -1)  node[circle,fill,inner sep=1pt]{};
	\draw (1, -1)  node[circle,draw, inner sep=3pt]{};
	\draw (2, -1)  node[circle,fill,inner sep=1pt]{};
	\draw (2, -1)  node[circle,draw, inner sep=3pt]{};
	\draw (6/5,-3/5)  node[circle,fill,inner sep=1pt]{};
	\draw (6/5,-3/5)  node[circle,draw, inner sep=3pt]{};
	\draw (2/5,-1/5)  node[circle,fill,inner sep=1pt]{};
	\draw (2/5,-1/5)  node[circle,draw, inner sep=3pt]{};
	\draw (3,-1)  node[circle,fill,inner sep=1pt]{};
	\draw (3,-1) node[circle,draw, inner sep=3pt]{};
	\draw (9/5,-3/5) node[circle,fill,inner sep=1pt]{};
	\draw (9/5,-3/5) node[circle,draw, inner sep=3pt]{};
	\draw (9/13,-6/13) node[circle,fill,inner sep=1pt]{};
	\draw (9/13,-6/13)  node[circle,draw, inner sep=3pt]{};
	\draw (1,0)  node[circle,fill,inner sep=1pt]{};
	\draw (1,0)  node[circle,draw, inner sep=3pt]{};
	\draw (3,0)  node[circle,fill,inner sep=1pt]{};
	\draw (3,0)  node[circle,draw, inner sep=3pt]{};
	\end{tikzpicture}
	\caption{\textit{Hilbert diagram}. The circled dots describe the indexing set $ \calB $. The two lines pass through the weights of strictly semistable points (see Proposition \ref{prop:toristrictly}).}\label{fig:1}
\end{figure}

There are 16 black nodes in this square, and each of these nodes represents a monomial $ x_0^ix_1^{3-i}y_0^jy_1^{3-i} $ in $ H^0(\PP^1\times\PP^1, \calO_{\PP^1\times\PP^1}(3, 3)) $, for $ 0\leq i, j \leq 3 $. This square is simply the diagram of weights $ \alpha_I=\alpha_{(i,j)} $ of the representation of $ G $ on $H^0(\PP^1\times\PP^1, \calO_{\PP^1\times\PP^1}(3, 3))$ with respect to the standard maximal torus $ T:=(\diag(a, a^{-1}), \diag(b, b^{-1}), 1) $ in $ G $. Each of the nodes denotes a weight of this representation, namely
\begin{equation}\label{eq:weight}
x_0^ix_1^{3-i}y_0^jy_1^{3-j} \leftrightarrow (3-2i, 3-2j), \ \mathrm{for} \ i,j=0, ..., 3.
\end{equation}
There is a non-degenerate inner product (the Killing form) defined in the Cartan subalgebra $ \mathfrak{t}:=\text{Lie}(T\cap (\SU(2, \CC)\times \SU(2, \CC))) $ in $ \text{Lie}(\SU(2, \CC)\times \SU(2, \CC))\otimes \CC \cong \text{Lie}(G) $. Using this inner product, we can identify the Lie algebra $ \mathfrak{t} $ with its dual $ \mathfrak{t}^{\vee} $, and the above square can be thought of as lying in $ \mathfrak{t} $. The axes of the Hilbert diagram thus coincide with the Lie algebras of the two factors of the maximal compact torus.

The Weyl group $ W(G):=N(T)/T\cong (\mm \times \mm)\ltimes \mm $ coincides with the dihedral group $ D_8 $ of all symmetries of the square. It operates on the \textit{Hilbert diagram} as follows: the first two involutions are reflections along the axes, while the third one is along the principal diagonal. It is easy to see that the grey region is the portion of the square which lies inside a fixed positive Weyl chamber $ \mathfrak{t}_+ $.

By definition, the indexing set $ \calB $ consists of vectors $ \beta $ such that $ \beta $ lies in the closure $ \bar{\mathfrak{t}}_+ $ of the positive Weyl chamber and is also the closest point to the origin of a convex hull spanned by a non-empty set of weights of the representation of $ G $ on $ H^0(\PP^1\times\PP^1, \calO_{\PP^1\times\PP^1}(3, 3)) $. In this situation, we may assume that such a convex hull is either a single weight or it is cut out by a line segment joining two weights, which will be denoted by $ \langle \beta \rangle $ (see Figure \ref{fig:1}).

The codimension $ d(\beta)$ of each stratum $ S_{\beta}\subset X $ is equal to (see \cite[3.1]{Kir89})
\begin{equation}\label{codim}
 d(\beta)= n(\beta)-\dim G/P_\beta, 
\end{equation} 
where $ n(\beta) $ is the number of weights $ \alpha_I $ such that $ \beta\cdot \alpha_I < ||\beta||^2 $, i.e. the number of weights lying in the half-plane containing the origin and defined by $ \beta $. Moreover, let $ P_\beta \subseteq G$ be the subgroup of elements in $ G $ which preserve $ \bar{Y}_{\beta} $, then $ P_{\beta} $ is a parabolic subgroup, whose Levi component is the stabiliser $ \Stab\beta $ of $ \beta\in \mathfrak{t} $ under the adjoint action of $ G $.

All the contributions coming from the unstable strata are summarised in Table \ref{tab:1} and were computed looking at the Figure \ref{fig:1}. 

\begin{table}[h]
	\centering
 \begin{tabular}{c|c|c|c|c}
	weights in $ \langle \beta \rangle $& $ n(\beta) $&$\Stab \beta $& $2d(\beta)$&$P_t^{G}(S_{\beta})$\\
	\hline \rule{0pt}{2.5ex}
	$ (3, -3) $&15&$\langle T, \iota \rangle$&26&$ (1-t^2)^{-1}(1-t^4)^{-1} $\\[1ex]
	$ (3, -1), (1, -3) $&13&$\langle T, \iota \rangle$&22&$ (1-t^2)^{-1} $\\[1ex]
	$ (3, 1), (1, -1), (-1, -3) $&10&$\langle T, \iota \rangle$&16&$ (1+t^2-t^6)(1-t^2)^{-1}(1-t^4)^{-1} $\\[1ex]
	$ (1, -3), (3, 1) $&12 & $T$ &20&$(1-t^2)^{-1}$\\[1ex]
	$ (3, 3), (1, -1) $&10 & $T$ &16&$(1-t^2)^{-1}$\\[1ex]
	$ (1, 1), (-1, -3) $&8 & $T$ &12&$(1-t^2)^{-1}$\\[1ex]
	$ (3, -1) $&14 & $T$ &24&$(1-t^2)^{-2}$\\[1ex]
	$ (1, -3), (3, 3) $&11 & $T$ &18&$(1-t^2)^{-1}$\\[1ex]
	$ (-1, -3), (3, 3) $&9 & $T$ &14&$(1-t^2)^{-1}$\\[1ex]
	$ (3, -3), (3, -1), (3, 1), (3, 3) $& 12 & $\CC^*\times \SL(2, \CC)$ &22&$(1-t^2)^{-1}$\\[1ex]
	$ (1, -3), (1, -1), (1, 1), (1, 3) $& 8 &$\CC^*\times \SL(2, \CC)$& 14&$(1-t^2)^{-1}$
\end{tabular}
\caption{Cohomology of the unstable strata.} \label{tab:1}
\end{table}
The element
 $$ \iota:=\left (\begin{pmatrix} 
0 & 1 \\
-1 & 0 
\end{pmatrix}, \begin{pmatrix} 
0 & -1 \\
1 & 0 
\end{pmatrix}, -1\right ) $$
 is a generator of  $\langle T, \iota \rangle \cong (\CC^*)^2\rtimes \mathbb{Z}_2 $, with automorphism $ (a, b)\leftrightarrow (b^{-1}, a^{-1}) $, which is a double cover of the maximal torus $ T $. For every $ \beta \in \calB $, the first column of Table \ref{tab:1} shows the weights contained in the segment $ \langle \beta \rangle $ orthogonal to the vector $ \beta\in \mathfrak{t} $ (see Figure \ref{fig:1}): then via the correspondence (\ref{eq:weight}) one can obtain an explicit geometrical interpretation of the curve contained in each unstable stratum. The terms appearing in the second, third and fourth columns are determined easily from the \textit{Hilbert diagram}. The computations in the last column follow from applying Theorem \ref{thm:equi} to the action of $ \Stab \beta $ on $ Z_{\beta} $, in order to compute the equivariant cohomology of each unstable stratum $ P_t^{\Stab \beta}(Z_{\beta}^{ss})=P_t^G(S_{\beta}) $ (see Remark \ref{rmk:unstable}). 
 
 We shall discuss some of these cases below.
\begin{lemma}
	There are exactly six unstable strata indexed by $\beta$, as listed in Table \ref{tab:1}, such that $Z_{\beta}\cong \PP^1$, and their equivariant Hilbert-Poincar\'{e} series is $P_t^G(S_{\beta})=(1-t^2)^{-1}$. 
\end{lemma}

\begin{proof}
	Looking at Figure \ref{fig:1}, there are 6 unstable strata indexed by $\beta \in \calB$ such that the segment $\langle \beta \rangle$ orthogonal to the vector $\beta$ contains two weights that generate the line $Z_{\beta}\subset X$. As summarised in Table \ref{tab:1}, in five of these cases the stabiliser $\Stab \beta$ is isomorphic to the maximal torus $T$ and hence by Remark \ref{rmk:unstable}
	\[ P_t^G(S_{\beta})=\frac{1+t^2}{(1-t^2)^2}-\frac{2t^2}{(1-t^2)^2}=\frac{1}{1-t^2}. \]
	In the remaining case, corresponding to the second row of Table \ref{tab:1}, the stabiliser is $\Stab \beta \cong \langle T, \iota \rangle$ and the cohomology of the corresponding stratum is 
	\[ P_t^G(S_{\beta})=\frac{1+t^2}{(1-t^2)(1-t^4)}-\frac{t^2}{(1-t^2)^2}=\frac{1}{1-t^2}. \]
\end{proof}

\begin{lemma}
	There is exactly one unstable stratum indexed by $\beta$, as listed in Table \ref{tab:1}, such that $Z_{\beta}\cong \PP^2$, and its equivariant Hilbert-Poincar\'{e} series is $P_t^G(S_{\beta})=(1+t^2-t^6)(1-t^2)^{-1}(1-t^4)^{-1}$. 
\end{lemma}

\begin{proof}
	The case under consideration corresponds to the third row of Table \ref{tab:1}, where the segment orthogonal to $ \beta $ contains three weights spanning $ Z_{\beta}\cong \PP^2 $. Hence and by Theorem \ref{thm:equi} the equivariant cohomological series of the correspondent stratum is
	\[ P_t^G(S_{\beta})=\frac{1+t^2+t^4}{(1-t^2)(1-t^4)}-\frac{t^4}{(1-t^2)^2}=\frac{1+t^2-t^6}{(1-t^2)(1-t^4)}. \]
\end{proof}

\begin{lemma}
	There are exactly two unstable strata indexed by $\beta$, as listed in Table \ref{tab:1}, such that $Z_{\beta}\cong \PP^3$, and their equivariant Hilbert-Poincar\'{e} series is $P^G(S_{\beta})=(1-t^2)^{-1}$.
\end{lemma}

\begin{proof}
	The cases under consideration correspond to the last two rows of Table \ref{tab:1}, where the segment orthogonal to $ \beta $ contains four weights spanning a $ \PP^3 $. The linear subspace $Z_{\beta}$ is acted on by the group $ \Stab \beta=\CC^* \times \SL(2, \CC) $. The first factor is central and acts trivially on $ Z_{\beta} $, while the action of the second factor can be identified with the action on the space $ \mathrm{Sym}^3 \PP^1 \cong \PP^3$ of binary cubic forms by change of coordinates. This leads to
	\[ P^G(S_{\beta})=P_t(B \CC^*) P_t^{SL(2, \CC)}((\mathrm{Sym}^3 \PP^1)^{ss})=P_t(B \CC^*)P(M_{0,3})=\frac{1}{1-t^2}. \]
\end{proof}

The remaining two cases, that is when $Z_{\beta}\cong \PP^0$, are easier and left to the reader.

We are finally ready to prove Proposition \ref{prop:seriesss}:
\begin{proof}[Proof of Proposition \ref{prop:seriesss}]
According to Theorem \ref{thm:equi}, we need to subtract all the contributions of the unstable strata, appearing in Table \ref{tab:1}, to the $ G $-equivariant cohomology of $ X $ computed in (\ref{eq:PGX}).
\end{proof}

\section{Kirwan blow-up}
\label{sec:blow}

In this Section we recall the general construction of the Kirwan blow-up of a GIT quotient which provides an orbifold resolution of singularities. It is achieved by stratifying the GIT boundary $X/\!\!/G \setminus X^s/\!\!/G$ in terms of the connected components $R$ of the stabilisers of the associated polystable orbits. Then, one proceeds by blowing-up these strata according to the dimension of the corresponding $R$. In our situation, the Kirwan blow-up $\widetilde{M} \rightarrow M$ is obtained by blowing-up three loci of strictly polystable points, geometrically described in Theorem \ref{thm:polystable} (see also Proposition \ref{prop:blow-up}).

\subsection{General setting}
\label{subsec:blow}
In general the equivariant cohomology $ H^*_G(X^{ss}) $ of the semistable locus does not coincide with the cohomology $ H^*(X/\!\!/G) $ of the GIT quotient, unless in the case when all semistable points are actually stable. This is not the case for us. The solution is given by constructing a \textit{partial resolution of singularities} $ \widetilde{X}/\!\!/G \rightarrow X/\!\!/G $, known as \textit{Kirwan blow-up} \cite{Kir85}, for which the group $ G $ acts with finite isotropy groups on the semistable points $ \widetilde{X}^{ss} $. We briefly describe how it is constructed.

We consider again the setting, as in Section \S \ref{subsec:strata}, of a smooth projective manifold $ X \subset \PP^n $ acted on by a  reductive group $ G $. We also assume throughout the paper that the stable locus $ X^{s}\neq \varnothing $ is non-empty. In order to produce the Kirwan blow-up, we need to study the GIT boundary $ X/\!\!/G \smallsetminus X^s/\!\!/G $ and stratify it in terms of the isotropy groups of the associated semistable points. More precisely, let $ \calR $ be a set of representatives for the conjugacy classes of connected components of stabilisers of strictly polystable points, i.e. semistable points with closed orbits, but infinite stabilisers. Let $ r $ be the maximal dimension of the groups in $ \calR $, and let $ \calR(r)\subseteq \calR $ be the set of representatives for conjugacy classes of subgroups of dimension $ r $. For every $ R\in \calR(r) $, consider the fixed locus
\begin{equation}
	Z_{R}^{ss}:=\{ x \in X^{ss} : R \ \text{fixes} \ x \} \subset X^{ss}.
\end{equation}

Kirwan showed \cite[\S 5]{Kir85} that the subset 
\[ \bigcup_{R\in \calR(r)}G\cdot Z_{R}^{ss}\subset X^{ss} \]
is a disjoint union of smooth $ G $-invariant closed subvarieties in $ X^{ss} $. Now let $ \pi_1:X_1\rightarrow X^{ss} $ be the blow-up of $ X^{ss} $ along $\bigcup_{R\in \calR(r)} G\cdot Z_{R}^{ss} $ and $ E\subset X_1 $ be the exceptional divisor.

Since the centre of the blow-up is invariant under $ G $, there is an induced action of $ G $ on $ X_1 $, linearised by a suitable ample line bundle. If $ L=\calO_{\PP^n}(1)|_{X} $ is the very ample line bundle on $ X $ linearised by $ G $, then there exists $ d\gg 0 $ such that $ L_1:=\pi_1^* L^{\otimes d} \otimes \calO_{X_1}(-E) $ is very ample and admits a $ G $-linearisation (see \cite[3.11]{Kir85}). After making this choice, the set $ \calR_1 $ of representatives for the conjugacy classes of connected components of isotropy groups of strictly polystable points in $ X_1 $ will be strictly contained in $ \calR $ (see \cite[6.1]{Kir85}). Moreover, the maximum among the dimensions of the reductive subgroups in $ \calR_1 $ is strictly less than $ r $. Now we restrict to the new semistable locus $ X_1^{ss}\subset X_1 $, so that we are ready to perform the same process as above again. 

After at most $ r $ steps, we obtain a finite sequence of \textit{modifications}:
\begin{equation}
	\widetilde{X}^{ss}:= X^{ss}_r \rightarrow ... \rightarrow X^{ss}_1 \rightarrow X^{ss},
\end{equation}
by iteratively restricting to the semistable locus and blowing-up smooth invariant centres (cf. \cite[6.3]{Kir85}).

Therefore, in the last step, $ \widetilde{X}^{ss} $ is equipped with a $ G $-linearised ample line bundle such that $ G $ acts with finite stabilisers. In conclusion, we have the diagram
\begin{equation}
\begin{tikzcd}
\widetilde{X}^{ss} \arrow{r} \arrow{d} & X^{ss} \arrow{d} \\
\widetilde{X}/\!\!/G \arrow{r}& X/\!\!/G,
\end{tikzcd}
\end{equation}
where the \textit{Kirwan blow-up} $ \widetilde{X}/\!\!/G $, having at most finite quotient singularities, gives a \textit{partial desingularization} of $ X/\!\!/G $, which in general has worse singularities.

\subsection{Kirwan blow-up for $ (3, 3) $ curves in $ \PP^1 \times \PP^1 $}
Coming back to our case, we need to find the indexing set $ \calR $ of the Kirwan blow-up and the corresponding spaces $ Z_{R}^{ss} $, for all $ R\in \calR $. Namely, one must compute the connected components of the identity in the stabilisers among all the four families of polystable curves listed in Theorem \ref{thm:polystable}. Compared to \cite[\S 2.2]{Fed12}, we provide a more explicit, but equivalent, way to find the indexing set $ \calR $, which has also the advantage to compute $ Z_{R} $ and $ Z_{R}^{ss} $.

We must find which non-trivial connected reductive subgroups $ R \subset G $ fix at least one semistable point. Firstly, since $ R $ is connected, $ R $ must be contained in $ G^0=\SL(2, \CC)\times \SL(2, \CC) $. Secondly, since we are interested only in the conjugacy class of $ R $, we may assume that its intersection $T_R:=R\cap T $ with the maximal torus is a maximal torus of $ R $ and $ R\cap (\SU(2, \CC)\times \SU(2, \CC)) $ is a maximal compact subgroup. Since $ 0\in \mathfrak{t} $ is not a weight, it follows that $ T\cong (\CC^*)^2 $ fixes no semistable points. Therefore $ T_ R $ is a subtorus of rank one. 

The fixed point set $ Z_R^{ss} $ in $X^{ss}  $ consists of all semistable points whose representatives in $ H^0(\PP^1\times\PP^1, \calO_{\PP^1\times\PP^1}(3, 3))\cong\CC^{16} $ are fixed by the linear action of $ R $. Thus $ \mathrm{Fix}(T_R, \CC^{16}) $ is spanned by those weight vectors which lie on a line through the centre of the \textit{Hilbert diagram} and orthogonal to the Lie subalgebra $ \text{Lie}(T_R\cap (\SU(2, \CC)\times \SU(2, \CC))) \subset \mathfrak{t} $. Up to the action of a suitable element of the Weyl group $ W(G) $, we can assume that the line passes through the chosen closed positive Weyl chamber $ \bar{\mathfrak{t}}_+ $. We have only two possibilities, see Figure \ref{fig:1}.

Therefore we proved the following
\begin{prop} \label{prop:toristrictly}
	If $ R\in \calR $ is a subgroup in the indexing set of Kirwan's partial resolution, let $ T_R $ denote the maximal torus of $ R $ and let $ Z_R^{ss} $ denote the fixed-point set of $ R $ in $ X^{ss} $. Then, up to conjugation, there are two possibilities for $ T_R $ and $ Z_R^{ss} $:
	\begin{enumerate}[(i)]
		\item $ T_R=T_1:=\{ (\diag(t, t^{-1}), \diag(t, t^{-1}), 1): t\in \CC^* \} $ and $ Z_R^{ss} $ is contained in the projective space $ \PP\{ ax_0^3y_1^3+bx_0^2x_1y_0y_1^2+cx_0x_1^2y_0^2y_1+dx_1^3y_0^3 \}\cong\PP^3 $ spanned by the polynomials $x_0^3y_1^3$, $x_0^2x_1y_0y_1^2$, $x_0x_1^2y_0^2y_1$ and $x_1^3y_0^3 $.
		\item $ T_R=T_2:=\{ (\diag(t, t^{-1}), \diag(t^3, t^{-3}), 1): t\in \CC^* \} $ and $ Z_R^{ss} $ is contained in the projective space $ \PP\{ ax_0^3 y_0y_1^2+bx_1^3y_0^2y_1 \}\cong \PP^1 $ spanned by the polynomials $ x_0^3 y_0y_1^2$ and $x_1^3y_0^2y_1$. 
	\end{enumerate}
\end{prop}

We start analysing the second case. We can easily see from the characterisation of semistable points (Theorem \ref{thm:polystable}) that all the semistable curves are given by $ y_0y_1(ax_0^3y_1+bx_1^3y_0) $ with $ a\neq 0 $ and $ b\neq 0 $.
Geometrically these curves contains two lines of the same ruling and the residual curve intersects them in 2 points, giving 2 singularities of type $ A_5 $. We will call these curves as \textit{A-curves}. Their singular points are $ ((0:1), (0:1)) $ and $ ((1:0), (1:0)) $ in $ \PP^1\times \PP^1 $; see the Figure \ref{fig:2A5}.
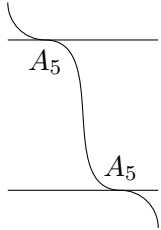
\begin{figure}[H]
\centering
\begin{tikzpicture}
\draw (-1, 1) -- (1, 1);
\draw (1, -1) -- (-1, -1);
\draw (-1, 1.5) to[out=270, in=180] (-0.5, 1);
\draw (-0.5, 1) to[out=0, in=180] (0.5, -1);
\draw (0.5, -1) to[out=0, in=90] (1, -1.5);
\node [below] at (-0.5,1) {$ A_5 $};
\node [above] at (0.5,-1) {$ A_5 $};
\end{tikzpicture}
\caption{Curve with $ 2A_5 $ singularities.}	\label{fig:2A5}
\end{figure}

By rescaling the variables $ x_0 $ and $ x_1 $, all the semistable \textit{A-curves} are equivalent to the curve $ C_{2A_5} $ defined by
\[ C_{2A_5}:=\{F_{C_{2A_5}}:=y_0y_1(x_0^3y_1+x_1^3y_0)=0 \}. \] 
Through this geometric description, it is now easy to show that in this case actually $ R=T_R $. We recall that $ R $ is the connected component of the identity in the stabiliser of the \textit{A-curves}: up to conjugation, we can think just of $ C_{2A_5} $. Yet every element of $ R $, stabilising the point corresponding to $ C_{2A_5} $ in $ X $, will induce an automorphism of $ C_{2A_5} $, which a fortiori must preserve the singular locus. Therefore every element of $ R $ must fix $ ((0:1), (0:1)) $ and $ ((1:0), (1:0)) $ or interchange them. Hence 
\[ R\subseteq T \sqcup \left \{ \left (\begin{pmatrix} 
0 & \alpha \\
-\alpha^{-1} & 0 
\end{pmatrix}, \begin{pmatrix} 
0 & \beta \\
-\beta^{-1} & 0 
\end{pmatrix}, 1\right ) : \alpha, \beta \in \CC^* \right \} \subseteq G. \]
From the connectedness of $ R $, it follows $ R\subseteq T $, hence $ R=T\cap R=T_R=T_2 $.

Now we analyse the first case. We can easily see via the Hilbert-Mumford numerical criterion (\cite[2.1]{MFK94}) that all the semistable curves are given by
\[ ax_0^3y_1^3+bx_0^2x_1y_0y_1^2+cx_0x_1^2y_0^2y_1+dx_1^3y_0^3=0, \]
where $ (a,b) $ are not simultaneously zero and $ (c, d) $ are not simultaneously zero, i.e. $ Z_{T_1}^{ss}=\PP^3\smallsetminus \{ a=b=0, \ c=d=0 \} $.  Moreover we can write every curve
\[ ax_0^3y_1^3+bx_0^2x_1y_0y_1^2+cx_0x_1^2y_0^2y_1+dx_1^3y_0^3=L_1L_2L_3; \]
\[ L_i=\alpha_ix_0y_1+\beta_ix_1y_0, \: (\alpha_i:\beta_i)\in\PP^1, \ i=1, 2, 3. \]
as the union of three conics in the class $ (1, 1) $, all meeting at points $ ((0:1), (0:1)) $ and $ ((1:0), (1:0)) $ in $ \PP^1\times \PP^1 $. We find three cases depending on how many $ L_i $'s coincide.
\begin{enumerate}[(i)]
	\item Assume that all the $ L_i $ coincide, namely the curve is a triple conic, which turns out to be equivalent to $ 3C $, defined by
	\[3C:= \{F_{3C}:=(x_0y_1-x_1y_0)^3=0\}. \]
	This curve is nothing but a triple line $ \PP^1\subset \PP^1\times \PP^1 $ diagonally embedded. Thus its stabiliser in $ \PGL(2, \CC)\times \PGL(2, \CC) $ is $ \PGL(2, \CC) $ diagonally embedded, too. We get a non-splitting central extension of groups:
	\begin{equation}\label{extension}
	1\rightarrow \mu_2 \times \mu_2 \rightarrow H \rightarrow \PGL(2, \CC) \rightarrow 1,
	\end{equation}
	where $ H:=\{ (A, \pm A): A\in \SL(2, \CC) \} $ is the stabiliser of $ 3C $ in $ G^0 $, that is to say the preimage of $ \PGL(2, \CC) $ under the natural homomorphism $ G^0=\SL(2, \CC) \times \SL(2, \CC)\rightarrow \PGL(2, \CC)\times \PGL(2, \CC) $. Here $ \mu_2 \times \mu_2 $ must be thought as the subgroup $ \{ (\pm I, \pm I), (\pm I, \mp I) \}\subset H $. Therefore we find the indexing subgroup $ R=\SL(2, \CC) $ diagonally embedded in $ G^0 $ and the associated spaces $ Z_{R}=Z_{R}^{ss}=\{ 3C \} $ are one point.
	\item Assume that two $ L_i $ coincide and the third one does not. The semistable curves of this type are unions of a smooth double conic and a conic that is nonsingular along the double conic. They intersect at the points $ ((0:1), (0:1)) $ and $ ((1:0), (1:0)) $, which consist of singularities of type $ D_8 $; see Figure \ref{fig:2D8}.
	
	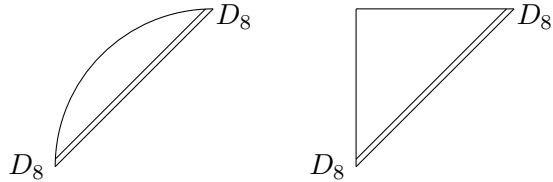
\begin{figure}[htbp]
		\centering
		\begin{tikzpicture}
		\draw (-5, -1) -- (-3, 1);
		\draw (-5, -1.1) -- (-2.9, 1);
		\draw (-5, -1.1) to[out=90, in=180] (-2.9, 1);
		\node [left] at (-5, -1.1) {$ D_8 $};
		\node [right] at (-3, 0.9) {$ D_8 $};
		
		\draw (-1, -1) -- (1, 1);
		\draw (-1, -1.1) -- (1.1, 1);
		\draw (-1, -1.1) -- (-1, 1);
		\draw (-1, 1) -- (1.1, 1);
		\node [left] at (-1, -1.1) {$ D_8 $};
		\node [right] at (1, 0.9) {$ D_8 $};
		\end{tikzpicture}
		\caption{Curves with $ 2D_8 $ singularities.} \label{fig:2D8}
	\end{figure}
	Now we can argue like in the case of $ C_{2A_5} $, noticing that every element of $ R $ must preserve the $ D_8 $ singular points. Therefore $ R\subseteq T $, so that $ R=T_1 $.
	
	\item Assume all the $ L_i $ are distinct from each other. The semistable curves of this kind are unions of three conics meeting in two $ D_4 $ singularities. These singular points are again $ ((0:1), (0:1)) $ and $ ((1:0), (1:0)) $; see Figure \ref{fig:2D4}.

	\begin{figure}[h]
		\centering
		\begin{tikzpicture}
		\draw (1, 1) -- (-1, -1);
		\draw (1, 1) to[out=180, in=90] (-1, -1);
		\draw (1, 1) to[out=-90, in=0] (-1, -1);
		\node [left] at (-1, -1) {$ D_4 $};
		\node [right] at (1, 1) {$ D_4 $};
		\end{tikzpicture}
		\caption{Curve with $ 2D_4 $ singularities.} \label{fig:2D4}
	\end{figure}
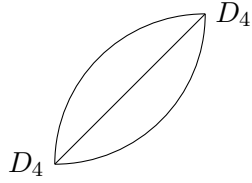
	Arguing once more as before, we find that $ R=T_1 $.
\end{enumerate}

	In conclusion, we proved the following:
	\begin{prop}\label{prop:blow-up}
		The indexing set $ \calR $ of the Kirwan blow-up, such as the fixed loci $ Z_R^{ss} $, for $ (3,3) $ curves in $ \PP^1 \times \PP^1 $, can be described as follows: 
		\begin{enumerate}[(i)]
			\item $ R_{C}:= \SL(2, \CC)$, diagonally embedded in $ G^0 $, and in this case $ Z_{R_{C}}=Z_{R_{C}}^{ss}=\{ 3C \} $ is the triple conic.
			\item $ R_D:=\{ (\diag(t, t^{-1}), \diag(t, t^{-1}), 1): t\in \CC^* \}\cong \CC^*$ and in this case $ Z_{R_D}^{ss}=\PP\{ ax_0^3y_1^3+bx_0^2x_1y_0y_1^2+cx_0x_1^2y_0^2y_1+dx_1^3y_0^3 \}\smallsetminus \{ a=b=0, c=d=0 \} $ is the set of D-curves.
			\item $ R_{A}:=\{ (\diag(t, t^{-1}), \diag(t^3, t^{-3}), 1): t\in \CC^* \}\cong \CC^* $ and in this case $ Z_{R_A}^{ss}=\PP\{ ax_0^3 y_0y_1^2+bx_1^3y_0^2y_1 \} \smallsetminus \{ a=0, b=0 \} $ is the set of \textit{A-curves}. 
		\end{enumerate}
	Moreover, the following holds:
	\[ R_D \subset R_{C}, \: R_{A}\cap R_{C}=\{ (\pm I, \pm I, 1)  \},\]
	\[  G\cdot Z_{R_{C}}^{ss}\subset G\cdot Z_{R_D}^{ss}, \: G\cdot Z_{R_{A}}^{ss} \cap G\cdot Z_{R_D}^{ss}=\emptyset. \]
	\end{prop}

We recall that the Kirwan's partial desingularization process consists of successively blowing-up $ X^{ss} $ along the (strict transforms of the) loci $ G\cdot Z_R^{ss} $ in order of dim$ R $, to obtain the space $ \widetilde{X}^{ss} $, and then taking the induced GIT quotient $ \widetilde{X}/\!\!/G $ with respect to a suitable linearisation. 

In our situation, we get the diagram
$$
\begin{tikzcd}
\widetilde{X}^{ss}=(\Bl_{G\cdot Z_{R_{A}}^{ss}}X_2^{ss})^{ss} \arrow{r} \arrow{r} \arrow{d}& X_2^{ss}=(\Bl_{G\cdot Z_{R_D,1}^{ss}}X_1^{ss})^{ss} \arrow{r} &X_1^{ss}=(\Bl_{G\cdot Z_{R_{C}}^{ss}}X^{ss})^{ss} \arrow{r} & X^{ss} \arrow{d} \\
\widetilde{M} \arrow{rrr}&&& M.
\end{tikzcd}
$$

The space $ \widetilde{X}^{ss} $ is obtained by firstly blowing up the orbit of the triple conic $ G\cdot Z_{R_{C}}^{ss} $, followed by the blow-up of $ G\cdot Z_{R_D,1}^{ss} $, namely the strict transform of the locus of \textit{D-curves} $ G\cdot Z_{R_D}^{ss} $ under the first bow-up. In the end we need to blow-up the orbit $ G\cdot Z_{R_{A}}^{ss} $ of $ C_{2A_5} $. We also observe that the third blow-up commutes with the other two, because the orbit of \textit{A-curves} is disjoint from the locus of \textit{D-curves}. Thus we find:

\begin{definition}\label{def:Mtilde}
The \textit{Kirwan blow-up} $ \widetilde{M}:=\widetilde{X}/\!\!/G \rightarrow M $ is defined as the GIT quotient of the blown-up variety $ \widetilde{X}^{ss} $ constructed above.
\end{definition}

Intrinsically at the level of moduli spaces, $ \widetilde{M} $ is obtained by first blowing up the point $ G\cdot Z_{R_{C}}^{ss}/\!\!/G $ corresponding to triple conics, then the strict transform $ \Bl_{G\cdot Z_{R_{C}}^{ss}/\!\!/G}(G\cdot Z_{R_D}^{ss} /\!\!/G) $ of the surface corresponding to the \textit{D-curves} and eventually blowing up the point $ G\cdot Z_{R_{A}}/\!\!/G $ of \textit{A-curves}. Nevertheless, for computational reasons, we will prefer the first description.

\section{Cohomology of the Kirwan blow-up}

This Section is devoted to the proof of

\begin{theorem}\label{thm:cohoblow}
	The Hilbert-Poincar\'{e} polynomial of the Kirwan blow-up $ \widetilde{M} $ is
	$$ P_t(\widetilde{M})=1+4t^2+7t^4+11t^6+14t^8+14t^{10}+11t^{12}+7t^{14}+4t^{16}+t^{18}.$$
\end{theorem}

In the first part of the Section, we recall the general theory to compute the Betti numbers of the Kirwan blow-up $\widetilde{X}/\!\!/G \rightarrow X/\!\!/G$ of a GIT quotient. Since $\widetilde{X}/\!\!/G$ has only finite quotient singularities, its rational cohomology coincides with the equivariant cohomology of the semistable locus $\widetilde{X}^{ss}$, which in turn can be computed from the equivariant cohomology of $X^{ss}$ corrected by an error term (see Theorem \ref{thm:cohkirblow}). This error term is divided into a main and extra contribution: the former takes into account the geometry of the centres of the blow-ups and the latter the action of $G$ on the exceptional divisors.

In the second and third part of this Section, we complete the computation of the Betti numbers of $\widetilde{M}$ by calculating the main and extra terms appearing in Theorem \ref{thm:cohkirblow} for our case. This concludes the proof of Theorem \ref{thm:cohoblow}.  
 
\label{sec:cohomology}
\subsection{General setting} \label{subsec:settingblowup} The effect of the desingularization on the equivariant Poincar\'{e} series is explained in \cite{Kir85}.  We consider again the setting, as in Section \S \ref{subsec:blow}, of a nonsingular projective variety $ X $ together with a linear action of a reductive group $ G $. Assume that $ R $ is a connected reductive subgroup with the property that the fixed point set $ Z_R^{ss}\subset X^{ss} $ is non-empty, but that $ Z_{R'}^{ss}=\varnothing $ for all higher subgroups $ R'\subset G $ of higher dimension than $ R $.

Let $ \pi:\hat{X}\rightarrow X^{ss}$ be the blow-up of $ X^{ss} $ along $ G\cdot Z_R^{ss} $. Then the equivariant cohomology of $ \hat{X} $ is related to that of the exceptional divisor $ E $ by
\begin{equation}\label{cohomologyblowup}
H^*_G(\hat{X})=H^*_G(X^{ss})\oplus H^*_G(E)/H^*_G(G\cdot Z_R^{ss})
\end{equation}
(see \cite[\S 4.6]{GH78}, \cite[7.2]{Kir85}). If $ \calN^R $ denotes the normal bundle to $ G\cdot Z_R^{ss} $ in $ X^{ss} $, then the equivariant cohomology of the exceptional divisor $ E=\PP \calN^R $ can be computed via a degenerating spectral sequence, namely
\[ H^*_G(E)=H^*_G(G\cdot Z_R^{ss})(1+...+t^{2(\rk \calN^R-1)}). \]
Kirwan proved (\cite[5.10]{Kir85}) that $ G\cdot Z_{R}^{ss} $ is algebraically isomorphic to $ G\times_{N(R)} Z_{R}^{ss} $, where $ N(R)\subset G $ is the normaliser of $ R$, hence we can compute
\begin{equation}\label{formula:rank}
\rk \calN^R=\dim X -\dim G\cdot Z_{R}^{ss}= \dim X-(\dim G+\dim Z_{R}^{ss} -\dim N(R))
\end{equation}
and
\[ H^*_G(G\cdot Z_R^{ss})=H^*_{N(R)}(Z_R^{ss}). \] 
Therefore from (\ref{cohomologyblowup}), it follows that
\[ P_t^G(\hat{X})=P_t^G(X^{ss})+P_t^{N(R)}(Z_R^{ss})(t^2+...+t^{2(\rk \calN^R-1)}).\]
If we consider the \textit{HKKN stratification} $ \{ S_{\beta} \}_{\beta \in \calB} $ associated to the induced action of $ G $ on $ \hat{X} $ (see Theorem \ref{thm:strata}), we can apply Theorem \ref{thm:equi} to deduce the equivariant Hilbert-Poincar\'{e} series of the semistable locus:
\begin{equation}
P_{t}^{G}(\hat{X}^{ss})=P_{t}^{G}(\hat{X})-\sum_{0\neq\beta \in \hat{\calB} }t^{2 \mathrm{codim}(\hat{S}_{\beta})}P_{t}^{G}(\hat{S}_\beta).
\end{equation}
To use this formula, we have to determine the indexing set $ \hat{\calB} $. For this, we choose a point $ x\in Z_R^{ss} $ and consider the normal vector space $ \calN_x^R $ to $ G\cdot Z_R^{ss} $ in $ X^{ss} $ at this point. Since the action of $ R $ on $ X^{ss} $ keeps this point $ x $ fixed, there is a natural induced representation $ \rho: R \rightarrow \GL(\calN_x^R) $ of $ R $ on this vector space. Let $ \calB(\rho) $ denote the indexing set of the stratification of the $ R $-action on the projective slice $ \PP \calN_x^R $. For each $ \beta' \in \calB(\rho) $, we have the subspaces $ Z_{\beta', \rho}$, $ Z_{\beta', \rho}^{ss}$ and $ S_{\beta', \rho} $ defined as in Section \S \ref{subsec:strata} but with respect to the action of $ R $ on $ \PP \calN_x^R $. 

In \cite[\S 7]{Kir85}, it is proven that $ \hat{\calB} $ can be identified with a subset of $ \calB(\rho) $. Given $ \beta \in \hat{\calB} $, the Weyl group orbit $ W(G) $ of $ \beta $ decomposes into a finite number of $ W(R) $ orbits. There is a unique $ \beta'\in \calB(\rho) $ in each $ W(R) $ orbit contained in the $ W(G) $ orbit of $ \beta $. We thus denote by $ w(\beta', R, G) $ the number of $ \beta' \in \calB(\rho) $ lying in the Weyl group orbit $ W(G)\cdot \beta $. 

For each $ \beta' \in \hat{\calB} \subset \calB(\rho) $, there is an $ (N(R)\cap \Stab \beta') $-equivariant fibration
\[ \pi:Z_{\beta', R}^{ss}:=Z_{\beta'}^{ss}\cap \pi^{-1}(Z_R^{ss})\rightarrow Z_R^{ss} \]
 with all fibres isomorphic to $ Z_{\beta', \rho}^{ss} $. As for each stratum $ \hat{S}_{\beta'} $, its codimension in $ \hat{X} $ is the same as the codimension of $ \hat{S}_{\beta', \rho} $ in $ \PP \calN_x^R $, denoted by $ d(\PP \calN^R, \beta') $ and its Hilbert-Poincar\'{e} series $ P^G_t(\hat{S}_{\beta'}) $ is the same as $ P_t^{N(R)\cap \Stab \beta}(Z_{\beta', R}^{ss}) $.

A repeated application of this argument leads to a formula to compute inductively the equivariant cohomology $ H^*_G(\widetilde{X}^{ss}) $ of the semistable locus $ \widetilde{X}^{ss} $, whose GIT quotient gives the \textit{Kirwan blow-up}. Since $ G $ acts on $ \widetilde{X}^{ss} $ with finite stabilisers, its equivariant Hilbert-Poincar\'{e} polynomial coincides with that of the partial desingularization $ \widetilde{X}/\!\!/G $. We summarise all the previous considerations under the following
\begin{theorem}\cite[7.4]{Kir85} \label{thm:cohkirblow}In the above setting,
	the cohomology of the Kirwan blow-up is given by:
	\[ P_t(\widetilde{X}/\!\!/G)=P_{t}^G(\widetilde{X}^{ss})=P_{t}^G(X^{ss})+\sum_{R\in\calR}A_{R}(t), \]
	where the error term $ A_{R}(t) $ can be divided into main and extra terms, as follows:
\begin{align}
A_R(t)=&P_t^{N}(Z_R^{ss})(t^2+...+t^{2(\rk \calN^R-1)}) \tag{Main term}\\
&-\sum_{0\neq \beta'\in\calB(\rho)}\frac{1}{w(\beta', R, G)}t^{2d(\PP\calN^R, \beta')}P_t^{N\cap \Stab \beta'}(Z_{\beta', R}^{ss}). \tag{Extra term}
\end{align}
\end{theorem}
 \begin{remark} \label{rmk:cohextra}(cf. \cite[7.2]{Kir85} and \cite[4.1 (4)]{Kir88}) If $ Z_{\beta', \rho}^{ss}=Z_{\beta', \rho} $, the spectral sequence of rational equivariant cohomology associated to the fibration $ \pi:Z_{\beta',R}^{ss}\rightarrow Z_R^{ss} $ degenerates and hence
 	\[ P_t^{N\cap \Stab \beta'}(Z_{\beta', R}^{ss})=P_t^{N\cap \Stab \beta'}(Z_{R}^{ss})\cdot P_t(Z_{\beta', \rho}). \]
 \end{remark}

Due to the role they play in the aforementioned results, we compute the normalisers of the reductive subgroups in $ \calR $.
\begin{prop} \label{prop:normal}The normalisers of the reductive subgroups in $ \calR=\{ R_{C}, R_{D}, R_{A} \} $ are given as follows
	\begin{enumerate}[(i)]
		\item $ N(R_{C})=H\rtimes \mm \subset G $, where $ H=\{ ((A, \pm A), 1): A\in \SL(2, \CC) \}$ fits into the central extension (\ref{extension}):
		\[  1\rightarrow \mu_2 \times \mu_2 \rightarrow H \rightarrow \PGL(2, \CC) \rightarrow 1, \]
		and the semidirect product structure descends from that of $ G $.
		\item $ N(R_D)=S\rtimes \mm \subset G $, where $ S $ is the subgroup of some generalised permutation matrices, namely
		
		\[ S= T \sqcup \left \{ \left (\begin{pmatrix} 
		0 & \alpha \\
		-\alpha^{-1} & 0 
		\end{pmatrix}, \begin{pmatrix} 
		0 & \beta \\
		-\beta^{-1} & 0 
		\end{pmatrix}, 1\right ) : \alpha, \beta \in \CC^* \right \} \subseteq G, \]
		and the semidirect product structure descends from that of $ G $.
		\item $ N(R_{A})=S $, as above.
	\end{enumerate}
\end{prop}
\begin{proof}
	The proof of (2) and (3) is straightforward from the definition of normaliser.
	
	In the case (1), we prove that the normaliser $ N' $ of $ R_{C} $ in $ G^0=\SL(2, \CC)\times \SL(2, \CC) $ is $ H $, then the statement will follow from the symmetry of $ R_{C} $. Since $ R_{C} $ has index two in $ H $, it is normal in $ H $, hence $ H\subseteq N' $. For the converse, we need:
	
	\textit{Claim: For every $ n\in N' $, there exists a $ g\in R_{C} $ with $ gn\in T\cap N' $, where $ T $ is the maximal torus.}
	
	Indeed, any element $ n\in N' $ must conjugate the standard maximal torus $ R_D \subset R_{C} $. Since all the maximal tori in $ R_{C} $ are conjugate under the action of $ R_{C} $, it follows that there must exists a $ g'\in R_{C} $ such that $ \tilde{n}=g'n $ fixes the maximal torus $ R_D $, that is to say $ \tilde{n} $ belongs to the normaliser $ S $ of $ R_D $ in $ G^0 $. If $ \tilde{n}:=g'n\in T $, just take $ g=g' $. Otherwise $ \tilde{n}\in \sigma T$, where
	\[  \sigma:= \left ( \begin{pmatrix} 
	0 & 1 \\ 
	-1 & 0 
	\end{pmatrix}, \begin{pmatrix} 
	0 & 1 \\
	-1 & 0 
	\end{pmatrix}, 1 \right ) \in R_{C}. \]
	In this case, take $ g=\sigma^{-1}g' $ and the claim is proven.  
	
	By a straightforward matrix computation, we have that $ T\cap N'\subset H $. Now we can prove that $ N'\subseteq H $. Indeed, for every element $ n\in N' $, there is a $ g\in R_{C} $ with $ gn\in T\cap N'\subset H $, by the \textit{Claim}. Therefore $ n\in g^{-1}H=H $.  
\end{proof}
\subsection{Main error terms} This subsection is devoted to computing the main error terms for all the three stages of the partial desingularization.
\subsubsection{Triple conic}
As we have seen, the first step in the \textit{Kirwan blow-up} process is to blow-up the locus corresponding to triple conics.
\begin{prop}\label{prop:main3C}
	For the group $ R_{C}\cong \SL(2, \CC) $, the main term of $ A_{R_C}(t) $ is given by
	\begin{align*}
	P_t^{N(R_{C})}(Z_{R_{C}}^{ss})(t^2+...+t^{2(\rk \calN^{R_{C}}-1)})&=\frac{t^2+...+t^{22}}{1-t^4}\\
	&\equiv t^2+t^4+2t^6+2t^8 \ \mod t^{10}.\\
	\end{align*} 
\end{prop}
\begin{proof}
	We saw in Proposition \ref{prop:blow-up} that $ Z_{R_{C}}^{ss} $ consists of a single point, and in Proposition \ref{prop:normal} the normaliser $ N(R_{C}) $ and in (\ref{formula:rank}) how to compute the rank of the normal bundle, leading to:
	\begin{center}
		$ H^*_{N(R_{C})}(Z_{R_{C}}^{ss})=H^*(BN(R_C))=H^*(B(H\rtimes \mm))=H^*(BH)^{\mm}, $\\
		$ \rk \calN^{R_{C}}= \dim X-(\dim G+\dim Z_{_{R_{C}}}^{ss} -\dim N(R_{C}))=12. $ 
	\end{center}
	
	Now we recall that $ H $ fits into the central extension (\ref{extension}):
	\[  1\rightarrow \mu_2 \times \mu_2 \rightarrow H \rightarrow \PGL(2, \CC) \rightarrow 1, \]
	hence $ H^*(BH)^{\mm}=H^*(B\PGL(2, \CC))^{\mm} $, with the induced $ \mm $-action. From the description of $ R_{C} $, we saw that this copy of $ \PGL(2, \CC) $ must be thought as diagonally embedded in $ \PGL(2, \CC)\times \PGL(2, \CC) $, and $ \mm $ simply interchanges the two factors, acting trivially on the diagonal. This means that 
	\[ H^*_{N(R_{C})}(Z_{R_{C}}^{ss})=H^*(B\PGL(2, \CC))^{\mm}=H^*(B\PGL(2, \CC))\]
	and $ P_t^{N(R_{C})}(Z_{R_{C}}^{ss})=P_t(B\PGL(2, \CC))=(1-t^4)^{-1} $.
\end{proof}

\subsubsection{D-curves} In the second step, we need to blow up the locus of \textit{D-curves}.

\begin{prop}\label{prop:mainD}
	For the group $ R_{D}\cong \CC^* $, the main term of $ A_{R_D}(t) $ is given by
	\begin{align*}
	P_t^{N(R_D)}(Z_{R_D, 1}^{ss})(t^2+...+t^{2(\rk \calN^{R_D}-1)})&=\frac{1+t^2}{1-t^2}(t^2+...+t^{14})\\
	&\equiv t^2+3t^4+5t^6+7t^8 \ \mod t^{10}.\\
	\end{align*} 
\end{prop}

\begin{proof} For brevity, write $ R=R_D $ and $ N=N(R_D)=S\rtimes \mm $ (see Proposition \ref{prop:blow-up} and \ref{prop:normal}). Recall that $ Z_{R, 1}^{ss} $ is the strict transform of $ Z_{R}^{ss} $ in $ X_1^{ss} $ under the first blow-up. We want to give an easier to handle geometric description of $ Z_{R, 1}^{ss} $. 
	
	We saw in Proposition \ref{prop:blow-up} that $ Z_{R}^{ss}=\PP\{ ax_0^3y_1^3+bx_0^2x_1y_0y_1^2+cx_0x_1^2y_0^2y_1+dx_1^3y_0^3 \}\smallsetminus \{ a=b=0, c=d=0 \} $. The centre of the first blow-up is the orbit of the triple conic $ 3C $ which intersects $ Z_{R}^{ss} $ along the twisted cubic
	\[ G\cdot 3C \cap Z_R^{ss}= C^{ss}=\{ (u^3:3u^2v:3uv^2:v^3): (u:v)\in \PP^1, \ u, v\neq 0 \}\subset Z_{R}^{ss}, \] 
	corresponding to the union of three conics that are actually coincident. Therefore
	\[ Z_{R, 1}^{ss}=(\Bl_{C^{ss}}Z_{R}^{ss})^{ss}, \]
	because we recall that, after taking the proper transform, one should restrict only to the semistable points in $ X_1 \rightarrow X$ for the induced action of $ G $. We want to stress that the Kirwan blow-up is a blow-up, followed by a restriction to the semistable points. Nevertheless, by \cite[1.9]{Kir86}, $ Z_{R, 1}^{ss} $ is the set of semistable points for the natural action of $ N/R $ on $ \Bl_{C^{ss}}Z_{R}^{ss} $. But every point of $ Z_{R}^{ss} $ is actually stable for the action of $ N/R $ and it will remain stable after the blow up (see \cite[3.2]{Kir85}). This means that every point of $ \Bl_{C^{ss}}Z_{R}^{ss} $ is indeed stable and, a fortiori semistable for $ N/R $, hence:
	\[ Z_{R, 1}^{ss}=\Bl_{C^{ss}}Z_{R}^{ss}. \]
	In conclusion, $ Z_{R, 1}^{ss} $ is the blow-up of $ \PP^3\smallsetminus \{ a=b=0, c=d=0 \} $ along the twisted cubic $ C^{ss} $ and we need to compute $P_t^{N}(Z_{R,1}^{ss})=P_t^N(\Bl_{C^{ss}}(\PP^3)^{ss}) $. 
	According to (\ref{cohomologyblowup}), the equivariant cohomology of the blow-up is related to the centre by the formula:
	\[ P_t^N(\Bl_{C^{ss}}(\PP^3)^{ss})=P_t^N((\PP^3)^{ss})+t^2P_t^N(C^{ss}). \]
	The action of $ N $ on $ C^{ss} $ is transitive and the stabiliser of $ 3C=(1:-3:3:-1) $ in $ N $ is $ (H\cap S)\rtimes \mm $, where
	\[ H\cap S=\left \{ \left (\begin{pmatrix} 
	 \lambda & 0 \\
	0 & \lambda^{-1} 
	\end{pmatrix}, \begin{pmatrix} 
	 \pm \lambda&0 \\
	0& \pm \lambda^{-1} 
	\end{pmatrix}, 1\right ) : \lambda \in \CC^* \right \}
	\sqcup \left \{ \left (\begin{pmatrix} 
	0 &\eta \\
	-\eta^{-1} & 0 
	\end{pmatrix}, \begin{pmatrix} 
	0 & \pm \eta \\
	\mp \eta^{-1} & 0 
	\end{pmatrix}, 1\right ) : \eta \in \CC^* \right \}\]
	so $ P_t^N(C^{ss})=P_t(B(H\cap S))^{\mm} $. The natural homomorphism $ \SL(2, \CC)\times \SL(2, \CC)\rightarrow \PGL(2, \CC)\times \PGL(2, \CC) $ induces a central extension:
	\[ 1\rightarrow \mu_2 \times \mu_2 \rightarrow H\cap S \rightarrow K \rightarrow 1, \]
	where $ K\subset \PGL(2, \CC)\times \PGL(2, \CC) $ is the image of $ H\cap S $. Here $ K $ has a structure of semidirect product $ \CC^*\rtimes \mathrm{S}_2 $, where $ \mathrm{S}_2 $ acts on $ \CC^* $ by inversion. Hence
	\begin{align*}
	H^*_N(C^{ss})&=H^*(B(H\cap S))^{\mm}\\
	&=H^*(BK)^{\mm}\\
	&=(H^*(B\CC^*)^{\mathrm{S}_2})^{\mm}\\
	&=(\QQ[c]^{\mathrm{S}_2})^{\mm}=\QQ[c^2],
	\end{align*}
	because $ \mathrm{S}_2 $ acts on $ H^2(B \CC^*)=\QQ\langle c \rangle $ by $ c\leftrightarrow -c $ and $ \mm $ does trivially. This means that $ P_t^N(C^{ss})=(1-t^4)^{-1} $.
	
	Now we compute $ P_t^N((\PP^3)^{ss}) $: we consider the action of $ N $ on $ \PP^3\cong Z_{R} $ and the usual equivariantly perfect stratification (Theorem \ref{thm:strata} and \ref{thm:equi}), giving 
	\[ P_t^N((\PP^3)^{ss})=P_t^N(\PP^3)-\sum_{0\neq\beta \in \calB }t^{2 d(\beta)}P_{t}^{\Stab \beta}(Z_\beta^{ss}). \]
	Firstly we compute $ P_t^N(\PP^3) $. Notice that $ N $ is disconnected, with connected component of the identity equal to $ N^0=T $, and $ \pi_0(N)=\mm \rtimes \mm =\mm \times \mm $. Since $ \mm \times \mm $ acts by linear transformation on $ \PP^3 $, it acts trivially on cohomology $ H^*(\PP^3)=\QQ[h]/(h^4) $ and hence
	\begin{align*}
	H^*_N(\PP^3)&=(H^*(\PP^3)\otimes H^*(BT))^{\mm \times \mm}\\
				&=\QQ[h]/(h^4)\otimes  \QQ[c_1, c_2]^{\mm \times \mm},
	\end{align*}
	where $ \deg(c_1)=\deg(c_2)=2 $ and the action of $ \mm \times \mm $ on $ H^2(BT)=\QQ\langle c_1, c_2 \rangle $ is represented by the matrices
	\[ \begin{pmatrix} 
	-1 & 0 \\
	0 & -1 
	\end{pmatrix} \ \text{and} \ \begin{pmatrix} 
	0 & 1 \\
	1 & 0 
	\end{pmatrix}. \]
	By Molien's formula, we find that $ P_t^N(\PP^3)=(1+t^2+t^4+t^6)(1-t^4)^{-2} $.
	
	Since the action of $ T $ on $ \PP^3 $ has weights
	\[ (3, -3), \ (1, -1), \ (-1, 1), \ (-3, 3), \]
	 which correspond to the weights on the antidiagonal of the \textit{Hilbert diagram} (Figure \ref{fig:1}) in the Lie algebra $ \mathfrak{t} $, the indexing set of this stratification is $ \calB=\{ (0,0), (1, -1), (3, -3) \} $. The real codimension of the strata are $ 2d((1, -1))=4 $ and $ 2d((3, -3))=6 $, while for both indices $ Z_{\beta}=Z_{\beta}^{ss}=\PP^0 $ and 
	\[ \Stab \beta =  \langle T, \iota \rangle \cong (\CC^*)^2\rtimes \mm, \] 
	as in Table \ref{tab:1}, so that by Molien's formula $ P_{t}^{\Stab \beta}(Z_\beta^{ss})=P_t(B\Stab \beta)=(1-t^2)^{-1}(1-t^4)^{-1} $.
	In conclusion, putting everything together, we get:
	\[ P_t^N(\Bl_{C^{ss}}(\PP^3)^{ss})=\frac{1+t^2+t^4+t^6}{(1-t^4)^2}-\frac{t^4+t^6}{(1-t^2)(1-t^4)}+\frac{t^2}{1-t^4}=\frac{1+t^2}{1-t^2}. \]
The result follows by computing the rank $ \rk \calN^R $, via the formula (\ref{formula:rank}).
\end{proof}
\subsubsection{A-curves} In the last step we need to blow up the locus of \textit{A-curves}. Recall that this locus remains unaltered after the first two resolutions.
\begin{prop}
	For the group $ R_{A}\cong \CC^* $, the main term of $ A_{R_A}(t) $ is given by
	\begin{align*}
	P_t^{N(R_{A})}(Z_{R_{A}}^{ss})(t^2+...+t^{2(\rk \calN^{R_{A}}-1)})&=\frac{t^2+...+t^{18}}{1-t^4}\\
	&\equiv  t^2+t^4+2t^6+2t^8 \ \mod t^{10}.\\
	\end{align*} 
\end{prop}
\begin{proof}
	 To compute $ P_t^{N(R_{A})}(Z_{R_{A}}^{ss}) $, we use the equality \cite[1.17]{Kir86}:
	\[ H_{N}^*(Z_{R_{A}}^{ss})=(H^*(Z_{R_{A}}/\!\!/ N^0(R_{A}))\otimes H^*(BR))^{\pi_0 N(R_{A})}. \]
	In our case the action of $ N(R_{A})^0=T $ on $Z_{R_{A}}^{ss} $ is transitive, hence
	\[ H^*(Z_{R_{A}}^{ss}/\!\!/T)=H^*(\text{point})=\QQ. \]
	Moreover $ \pi_0 N(R_{A})=\mm $ acts on $ R\cong \CC^* $ by inversion, so that
	\begin{align*}
	 H^*_{N(R_{A})}(Z_{R_{A}}^{ss})&= (H^*(Z_{R_{A}}^{ss}/\!\!/T)\otimes H^*(B\CC^*))^{\mm} \\
	 &=(\QQ\otimes \QQ[c])^{\mm}\\
	 &=\QQ[c^2],
	\end{align*}
	where $ \deg(c)=2 $ and the $ \mm $ operates on $ \QQ[c] $ by $ c\leftrightarrow -c $. Hence $ P_t^{N(R_{A})}(Z_{R_{A}}^{ss})=(1-t^4)^{-1} $. The result follows by computing the rank $ \rk \calN^{R_{A}} $, via formula (\ref{formula:rank}).
\end{proof}
\subsection{Extra terms}

To complete the computation of the contribution $ A_R(t) $, we need to calculate the extra terms, as stated in Theorem \ref{thm:cohkirblow}. The crucial point is to analyse for each $ R\in \calR $ the representation $ \rho: R \rightarrow \Aut (\calN_x^R) $ on the normal slice to the orbit $ G\cdot Z_R^{ss} $ at a generic point $ x\in Z_R^{ss} $. Since here we are dealing only with a local geometry around $ x $, we can restrict to consider the normal slice to the orbit $ G^0\cdot Z_R^{ss} $, which is the connected component of $ G\cdot Z_R^{ss} $ at $ x $.

\subsubsection{Tangent space to orbits for hypersurfaces in $ \PP^1 \times \PP^1 $}
If $ F\in H^0(\PP^1\times \PP^1, \calO_{\PP^1\times \PP^1}(d, d)) $ is a bihomogeneous form of bidegree $ (d,d) $, it will define a hypersurface $ V(F)\subset \PP^1\times \PP^1 $. We wish to describe the tangent space to the orbit $ \GL(2, \CC)\times \GL(2, \CC)\cdot F $. We are actually interested in the normal space to the orbit: 
$$ \SL(2, \CC)\times \SL(2, \CC)\cdot \{ V(F) \} \subset \PP H^0(\PP^1\times \PP^1, \calO_{\PP^1\times \PP^1}(d, d)). $$ 
However, since the normal space of any submanifold $ Y $ in a projective space $ \PP(W) $ can, via the Euler sequence, be identified with the normal space to its cone $ C(Y) \subset W $, we can alternatively study the $ \GL(2, \CC)\times \GL(2, \CC) $-orbit of $ F $ in $ H^0(\PP^1\times \PP^1, \calO_{\PP^1\times \PP^1}(d, d)) $, rather than the $ \SL(2, \CC)\times \SL(2, \CC) $-orbit of $ V(F) $ in $ \PP H^0(\PP^1\times \PP^1, \calO_{\PP^1\times \PP^1}(d, d)) $.

The strategy to compute the tangent space to the $ \GL(2, \CC)\times \GL(2, \CC) $-orbit of $ F $ is to work with the Lie algebra $ \mathfrak{gl}(2, \CC)\times \mathfrak{gl}(2, \CC) $ and use the exponential map $ \exp: \mathfrak{gl}(2, \CC)\times \mathfrak{gl}(2, \CC)\rightarrow \GL(2, \CC)\times \GL(2, \CC) $. Given an element $ e\in \mathfrak{gl}(2, \CC)\times \mathfrak{gl}(2, \CC) $, the derivative $ \frac{d}{dt}(\exp(te)F)_{t=0} $ gives a vector in the tangent space to the orbit $ \GL(2, \CC)\times \GL(2, \CC)\cdot F $. If we take a basis of $ \mathfrak{gl}(2, \CC)\times \mathfrak{gl}(2, \CC) $, we then obtain generators for the tangent space to the orbit $ \GL(2, \CC)\times \GL(2, \CC)\cdot F $. In practice, we choose the elementary matrices $ e_{ij}^1=(\delta_{ij})_{i,j=1,2 } $ as a basis of $ \mathfrak{gl}(2, \CC)\times 0 $ and $ e_{ij}^2=(\delta_{ij})_{i,j=1,2}$ for $ 0\times \mathfrak{gl}(2, \CC) $, then we indicate 
\[ (DF)^k_{ij}:=\frac{d}{dt}(\exp(te_{ij}^k)F)|_{t=0}, \quad 1\leq i, j, k\leq 2.\]
In conclusion, the tangent space to the orbit $ \GL(2, \CC)\times \GL(2, \CC)\cdot F $ is spanned by the entries of the matrix $DF=((DF)_{ij}^1|(DF)_{ij}^2)_{i,j=1,2} $.

Coming back to our situation, we carry this procedure out for the equations of strictly polystable hypersurfaces of $ \PP^1\times \PP^1 $ of bidegree $ (3,3) $. Indeed the tangent space to the orbit $ \GL(2, \CC)\times \GL(2, \CC)\cdot F $ is given by the span of the entries of the following matrices:
\begin{enumerate}[(i)]
	\item For $ F=ax_0^3y_1^3+bx_0^2x_1y_0y_1^2+cx_0x_1^2y_0^2y_1+dx_1^3y_0^3 $, the matrix $ DF=(DF^1|DF^2) $ is given by
	\begin{equation}\label{normal1D}
	DF^1=\left( 
	\begin{array}{cc}
	3ax_0^3y_1^3+2bx_0^2x_1y_0y_1^2+cx_0x_1^2y_0^2y_1 & 3ax_0^2x_1y_1^3+2bx_0x_1^2y_0y_1^2+cx_1^3y_0^2y_1\\
	bx_0^3y_0y_1^2+2cx_0^2x_1y_0^2y_1+3dx_0x_1^2y_0^3&
	bx_0^2x_1y_0y_1^2+2cx_0x_1^2y_0^2y_1+3dx_1^3y_0^3\\
	\end{array} \right),
	\end{equation}
	\begin{equation}\label{normal2D}
	DF^2=\left( 
	\begin{array}{cc}
	bx_0^2x_1y_0y_1^2+2cx_0x_1^2y_0^2y_1+3dx_1^3y_0^3&
	bx_0^2x_1y_1^3+2cx_0x_1^2y_0y_1^2+3dx_1^3y_0^2y_1\\
	3ax_0^3y_0y_1^2+2bx_0^2x_1y_0^2y_1+cx_0x_1^2y_0^3&
	3ax_0^3y_1^3+2bx_0^2x_1y_0y_1^2+cx_0x_1^2y_0^2y_1
	\end{array} \right).
	\end{equation}
	The set of linear relations satisfied by the entries of $ DF $ is
	\begin{equation}\label{eqD}
	\begin{cases}
	(DF)^1_{11}=(DF)^2_{22};\\
	(DF)^1_{22}=(DF)^2_{11}.
	\end{cases}
	\end{equation}
	\item For $ F=F_{3C} $, which corresponds to $ (a:b:c:d)=(1:-3:3:-1) $ in the above equations, the set of linear relations satisfied by the entries of $ DF_{3C} $ consists of the previous ones with the two further relations:
	\begin{equation}\label{eq3C}
	\begin{cases}
	(DF_{3C})^1_{12}+(DF_{3C})^2_{12}=0;\\
	(DF_{3C})^1_{21}+(DF_{3C})^2_{21}=0.
	\end{cases}
	\end{equation}
	\item For $ F=F_{C_{2A_5}}=y_0y_1(x_0^3y_1+x_1^3y_0), $ the matrix $ DF_{C_{2A_5}} $ is given by
	\begin{equation}\label{normal2A5}
	DF_{C_{2A_5}}=\left( 
	\begin{array}{cc|cc}
	3x_0^3y_0y_1^2 & 3x_0^2x_1y_0y_1^2 & x_0^3y_0y_1^2+2x_1^3y_0^2y_1&x_0^3y_1^3+2x_1^3y_0y_1^2 \\
	3x_0x_1^2y_0y_1^2 & 3x_1^3y_0^2y_1 &2x_0^3y_0^2y_1+x_1^3y_0^3& 2x_0^3y_0y_1^2+x_1^3y_0^2y_1
	\end{array}
	\right).
	\end{equation}
	The set of linear relations satisfied by the entries of $ DF_{C_{2A_5}} $ is
	\begin{equation}\label{eq2A5}
	\begin{cases}
	(DF)^1_{11}+(DF)^1_{22}=(DF)^2_{11}+(DF)^2_{22};\\
	(DF)^1_{22}-(DF)^1_{11}=3(DF)^2_{11}-3(DF)^2_{22}.
	\end{cases}
	\end{equation} 
\end{enumerate}

\subsubsection{Triple conic} We compute the extra contribution coming from the blow up of the triple conic.
\begin{prop}\label{prop:extra3C}
	For the group $ R_{C}\cong \SL(2, \CC) $ the extra term of $ A_{R_C}(t) $ is given by
	\begin{align*}
		\sum_{0\neq \beta'\in\calB(\rho)}\frac{1}{w(\beta', R_{C}, G)}t^{2d(\PP\calN^{R_{C}}, \beta')}P_t^{N(R_{C})\cap \Stab\beta'}(Z_{\beta', R_{C}}^{ss})&=\frac{t^{12}(1+t^2+t^4+t^6+t^8)}{1-t^2}\\
		&\equiv 0 \ \mod t^{10}.
	\end{align*} 
\end{prop}

Firstly this lemma describes the weights of the representation $ \rho : R_{C} \rightarrow \Aut(\calN_x^{R_C}) $, where $ x=3C $.

\begin{lemma}\label{lem:weights3C}
	For $ R_{C}\cong \SL(2, \CC) $, $ \dim \calN_x^{R_{C}}=12 $, the weights of the representation $ \rho $ of $ R_{C} $ on $ \calN_x^{R_{C}} $ are as follows with the respective multiplicities
	\[ (\pm 6)\times 1, (\pm 4)\times 2, (\pm 2)\times 2, (0)\times 2. \]
\end{lemma}
\begin{proof}
	The maximal torus $T_1=\{ (\diag(t, t^{-1}), \diag(t, t^{-1}), 1) \}$ in $ R_{C} $ acts on the coordinates $ ((x_0:x_1), (y_0:y_1)) $ diagonally. Thus  each monomial is an eigenspace for the action of $ T_1 $. Hence $ H^0(\calO_{\PP^1\times \PP^1}(3,3))=\CC^{16} $ decomposes as a sum of one-dimensional representations of $ T_1 $ with the following multiplicities of weights
	\[ (\pm 6)\times 1, (\pm4)\times 2, (\pm2)\times 3, (0)\times 4. \]
	The tangent space to the orbit $ G\cdot C_{3C} $ is generated by the entries of the matrices (\ref{normal1D}) and (\ref{normal2D}) at $ 3C $. Each polynomial spans an eigenspace for the action of $ T_1 $ with weight equal to
	\[ (\pm 2) \times 2, (0)\times 4. \]
	Now the relations (\ref{eqD}) are among the weight $ 0 $ generators, thus we may drop two of them in forming a basis of the tangent space. The two further relations (\ref{eq3C}) are among generators of weights $ 2 $ and $ -2 $, respectively, so we can drop one generator of weight $ 2 $ and $ -2 $. In total, the weights on the tangent space to the orbit are given by
	\[ (\pm 2) \times 1, (0)\times 2. \]
	By subtracting the weights of the representation of the tangent space to the orbit from the weights of the representation of $ T_1 $ on $ \CC^{16} $, we obtain the weights of the action on the normal space.
\end{proof}
\begin{proof}[Proof of Proposition \ref{prop:extra3C}]
	From the description of the weights of $ \rho $ in the Lemma \ref{lem:weights3C}, we see that we can take $ \calB(\rho)=\{ 0, 2, 4, 6 \} $. We can compute the codimension of the strata $ Z_{\beta', R_{C}}^{ss} $ by means of the formula (\ref{codim}):
	\[ d(\PP\calN_x^{R_{C}}, \beta')=n(\beta')-\dim(R_{C}/P_{\beta'}), \]
	where $ n(\beta') $ is the number of weights less than $ \beta' $ and $ P_{\beta'} $ is the associated parabolic subgroup of dimension 2. After noticing that for every weight, $ w(\beta', R_{C}, G)=1 $ and $N(R_{C})\cap \Stab\beta'=\hat{T_1}\rtimes \mm$, where $ \hat{T_1}:=\{ (\diag(t, t^{-1}), \diag(\pm t, \pm t^{-1}), 1): t\in \CC^* \} $ is a double cover of $ T_1 $, the result follows.
\end{proof}
\subsubsection{D-curves} We compute the extra contribution coming from the blow up of the \textit{D-curves}.
\begin{prop}\label{prop:extraD}
	For the group $ R_{D}\cong \CC^* $, the extra term of $ A_{R_D}(t) $ is given by
	\begin{align*}
	\sum_{0\neq \beta'\in\calB(\rho)}\frac{1}{w(\beta', R_{D}, G)}t^{2d(\PP\calN^{R_{D}}, \beta')}P_t^{N(R_{D})\cap \Stab\beta'}(Z_{\beta', R_{D}}^{ss})&= \frac{(1+t^2)^2}{1-t^2}(t^8+t^{10}+t^{12}+t^{14})\\
	&\equiv t^8 \ \mod t^{10}.
	\end{align*}
\end{prop}

This lemma describes the weights of the representation $ \rho : R_{D} \rightarrow \Aut(\calN_x^{R_{D}}) $. Here $ x\in Z_{R_D}^{ss} $ is a general point: for our purposes it is enough to take it away from the locus of triple conics, but to fix an explicit point we consider $ x=V(F':= x_0^3y_1^3+x_1^3y_0^3) $.

\begin{lemma}\label{lem:weightsD}
	For $ R_{D}\cong \CC^* $, $ \dim \calN_x^{R_{C}}=8 $, the weights of the representation $ \rho $ of $ R_{D} $ on $ \calN_x^{R_D} $ are
	\[  (\pm 6)\times 1, (\pm 4)\times 2, (\pm2)\times 1. \]
\end{lemma}
\begin{proof}
	The vector space $ H^0(\calO_{\PP^1 \times \PP^1}(3, 3))=\CC^{16} $ decomposes as a sum of one-dimensional representations of $ R_{D} $ with the same multiplicities of weights as in the previous case:
	\[ (\pm 6)\times 1, (\pm4)\times 2, (\pm2)\times 3, (0)\times 4. \]
	The tangent space to the orbit $ \GL(2, \CC)\times \GL(2, \CC)\cdot F' $ is generated by the entries of the matrices (\ref{normal1D}) and (\ref{normal2D}), with $ a,d=1 $ and $ b, c=0 $. Each polynomial spans an eigenspace for the action of $ R_D $ with weights equal to
	\[ (\pm 2) \times 2, (0)\times 4. \]
	Now the relations (\ref{eqD}) are among the weight 0 generators, thus we may drop two of them in forming a basis of the tangent space. In total, the weights for $ R_D $ on the tangent space to the orbit $ \GL(2, \CC)\times \GL(2, \CC)\cdot F' $ are given by
	\[ (\pm 2) \times 2, (0)\times 2. \]
	However, we are interested in the normal space $ \calN_x^{R_{D}} $ to the orbit $ G\cdot Z_{R_D}^{ss} $. We know that $ Z_{R_D}^{ss}/\!\!/N $ is two-dimensional, thus the tangent space $ T_x(G\cdot Z_{R_D}^{ss}) $, when lifted to $ \CC^{16} $, is the sum of $ T_{F'}(\GL(2, \CC)\times \GL(2, \CC)\cdot F') $ together with two tangent vectors representing the direction along $ Z_{R_D}^{ss}/\!\!/N $. This two further vectors can be thought as coming from varying $ A $ and $ B $ around $ 0 $ in the equation (cf. \cite[p. 5658]{Fed12}):
	\[ F_{A,B}=x_0^3y_1^3+Ax_0^2x_1y_0y_1^2+Bx_0x_1^2y_0^2y_1+x_1^3y_0^3. \]
	The derivatives in these directions are $ \frac{d}{dA}F_{A, B}=x_0^2x_1y_0y_1^2 $ and $ \frac{d}{dB}F_{A, B}=x_0x_1^2y_0^2y_1 $, which, as expected, are of weight 0 and do not lie in the span of the weight-0 space of the orbit. Thus the lift to $ \CC^{16} $ of the tangent space to the orbit $ G\cdot Z_{R_D}^{ss} $ is given by a space with weights
	\[ (\pm 2) \times 2, (0)\times 4. \]
	By subtracting the weights of the representation of the tangent space to the orbit from the weights of the representation of $ R_D $ on $ \CC^{16} $, we obtain the weights of the action on the normal space.
\end{proof}
\begin{proof}[Proof of Proposition \ref{prop:extraD} ]
	From the description of the weights of $ \rho $ in the Lemma \ref{lem:weightsD}, we see that we can take $ \calB(\rho)=\{ \pm 6, \pm 4, \pm 2, 0 \} $. We can compute the codimension of the strata $ Z_{\beta', R_{D}}^{ss} $ via the formula (\ref{codim}).
	\[ d(\PP\calN_x^{R_{D}}, \beta')=n(\beta')-\dim(R_{D}/P_{\beta'}), \]
	where $ n(\beta') $ is the number of weights $ \alpha $ such that $ \alpha \cdot \beta'<||\beta'||^2 $ and $ P_{\beta'} $ is the associated parabolic subgroup. Due to the symmetry, the coefficient for every weight is $ w(\beta', R_D, G)=2 $ and, according to Remark \ref{rmk:cohextra}
	\[ P_t^{N(R_{D})\cap \Stab\beta'}(Z_{\beta', R_{D}}^{ss})=P_t^{N(R_{D})\cap \Stab\beta'}(Z_{R_D,1}^{ss})P_t(Z_{\beta', \rho}). \]
	because $ Z_{\beta', \rho}=Z_{\beta', \rho}^{ss} $ is either $ \PP^0 $ or $ \PP^1 $. One can easily compute the stabiliser $ \Stab\beta'=T\rtimes \mm \subset N(R_D)$, where the semidirect product is induced from $ G $. Arguing analogously to the main term of $ R_D $ (see Proposition \ref{prop:mainD}), one finds that:
	\[ P_t^{T\rtimes \mm}(Z_{R_D,1}^{ss})=\frac{(1+t^2)^2}{1-t^2}, \]
	completing the proof.
\end{proof}
\subsubsection{A-curves} We compute the extra contribution coming from the blow up of the \textit{A-curves}.
\begin{prop}\label{prop:extra2A5}
	For the group $ R_{A}\cong \CC^* $, the extra term of $ A_{R_A}(t) $ is given by
	\begin{align*}
	\sum_{0\neq \beta'\in\calB(\rho)}\frac{1}{w(\beta', R_{A}, G)}t^{2d(\PP\calN^{R_{A}}, \beta')}P_t^{N(R_{A})\cap \Stab\beta'}(Z_{\beta', R_{A}}^{ss})&=\frac{t^{10}+t^{12}+t^{14}+t^{16}+t^{18}}{1-t^2}\\
	&\equiv 0 \ \mod t^{10}.
	\end{align*}
	
\end{prop}

 The proof of the proposition will consists of showing that the codimension $ d(\PP\calN^{R_A}, \beta') $ of any stratum $ S_{\beta'}(\rho) $ for $ 0\neq \beta'\in \calB(\rho) $ is at least 5. Firstly this lemma describes the weights of the representation $ \rho : R_{A} \rightarrow \Aut(\calN_x^{R_{A}}) $.

\begin{lemma}\label{lem:weights2A5}
	For $ R_{A}\cong \CC^* $, $ \dim \calN_x^{R_{C}}=10 $, the weights of the representation $ \rho $ of $ R_{A} $ on $ \calN_x^{R_A} $ are
	\[ (\pm 12)\times 1, (\pm 10)\times 1, (\pm 8)\times 1, (\pm 6)\times 1, (\pm 4)\times 1. \]
\end{lemma}
\begin{proof}
	Recall that $ x $ is a general point of $ Z_{R_{A}}^{ss} $, but since $ G\cdot Z_{R_{A}}^{ss}=G\cdot C_{2A_5}$ we can take $ x=C_{2A_5} $. Hence to describe $ \calN_x^{R_A} $, we must simply describe the normal space to the orbit $ G\cdot C_{2A_5} $ at $C_{2A_5} $.
	
	The vector space $H^0(\calO_{\PP^1\times \PP^1}(3,3))= \CC^{16} $ decomposes as a sum of one-dimensional representation of $ R_{A} $ with the following multiplicities of weights
	\[ (\pm 12)\times 1, (\pm 10)\times 1, (\pm 8)\times 1, (\pm 6)\times 2, (\pm 4)\times 1, (\pm 2) \times 1, (0)\times 2. \]
	The tangent space to the orbit $ G\cdot C_{2A_5} $ is generated by the entries of the matrix (\ref{eq2A5}). Each polynomial spans an eigenspace for the action of $ R_{A} $ with weight equal to
	\[ (\pm 6)\times 1, (\pm 2) \times 1, (0)\times 4. \]
	Now the relations (\ref{eq2A5}) are among the weight 0 generators, thus we may drop two of them in forming a basis of the tangent space. In total, the weights for $ R_{A} $ on the tangent space to the orbit are given by
	\[ (\pm 6)\times 1, (\pm 2) \times 1, (0)\times 2. \]
	By subtracting the weights of the representation of the tangent space to the orbit from the weights of the representation of $ R_A $ on $ \CC^{16} $, we obtain the weights of the action on the normal space.
\end{proof}
\begin{proof}[Proof of Proposition \ref{prop:extra2A5} ]
	From the description of the weights of $ \rho $ in the Lemma \ref{lem:weights2A5}, we see that we can take $ \calB(\rho)=\{ \pm 12, \pm 10, \pm 8, \pm 6, \pm 4, 0 \} $. We can calculate the codimension via (\ref{codim})
	\[ d(\PP\calN_x^{R_A}, \beta')=n(\beta')-\dim(R_{A}/P_{\beta'}), \]
	where $ n(\beta') $ is the number of weights $ \alpha $ such that $ \alpha \cdot \beta'<||\beta'||^2 $ and $ P_{\beta'} $ is the associated parabolic subgroup, in this case equal to $ R_{A} $ since $ R_{A} $ is a torus. After noticing that for every non-zero weight, $ w(\beta', R_{A}, G)=2 $ and $N(R_{A})\cap \Stab\beta'=T$, the result follows.
\end{proof}
\subsection{Cohomology of $ \widetilde{M} $} We complete the proof of Theorem \ref{thm:cohoblow}.
\begin{proof}[Proof of Theorem \ref{thm:cohoblow}]
	From Theorem \ref{thm:cohkirblow}, we need to put all the previous results together to find the Betti numbers of the Kirwan partial desingularization $ \widetilde{M} $. For the sake of readability, we report only the polynomials modulo $t^{10}$, but one can double-check the result with the entire Hilbert-Poincar\'{e} series and observe that Poincar\'{e} duality effectively holds.
	\begin{align*}
	P_t(\widetilde{M})&=P_t^G(\widetilde{X}^{ss})\equiv   \\
	&1+t^2+2t^4+2t^6+4t^8 \tag{Semistable locus}\\
	&+t^2+t^4+2t^6+2t^8-0 \tag{Error term for \textit{triple conic}} \\
	&+t^2+3t^4+5t^6+7t^8-t^8 \tag{Error term for \textit{$ D $-curves}}\\
	&+t^2+t^4+2t^6+2t^8-0\tag{Error term for \textit{$ A$-curves}}\\
	&\equiv 1+4t^2+7t^4+11t^6+14t^8 \ \mod t^{10}.
	\end{align*}
\end{proof}

\section{Intersection cohomology of the moduli space $ M $}
\label{sec:intersection}

In this Section, we compute the intersection cohomology of $M$ descending from $\widetilde{M}$, and thus prove the following:

\begin{theorem}\label{thm:intM}
	The intersection Hilbert-Poincar\'{e} polynomial of $ M $ is
	$$ IP_t(M)=1+t^2+2t^4+2t^6+3t^8+3t^{10}+2t^{12}+2t^{14}+t^{16}+t^{18}.$$
\end{theorem}

In the first part of the Section, we recall Kirwan's procedure to compare the cohomology of $ \widetilde{X}/\!\!/G $ and the intersection cohomology of $ X/\!\!/G $, as explained in \cite{Kir86}. This is in turn an application of the \textit{Decomposition Theorem} by Be\u{\i}linson, Bernstein, Deligne and Gabber (cf. \cite{BBD82}).

In the second part of the Section, instead of applying the Decomposition Theorem directly to the blow-down map $ \widetilde{M}\rightarrow M $ at the level of GIT quotients, we follow Kirwan's results (see \cite{Kir86}) and study the variation of the intersection Betti numbers at the level of the parameter spaces $ X^{ss} $ and $ \widetilde{X}^{ss}$, under each stage of the resolution.
 
\subsection{General setting}
We start with the general setting, as in Section \S \ref{subsec:blow} and \S\ref{subsec:settingblowup}, of a projective manifold $ X $ acted on by a reductive group $ G $. We suppose to have performed all the stages of the modification $ \widetilde{X}^{ss}\rightarrow X^{ss} $, indexed by the set $ \calR $, so that the \textit{Kirwan blow-up} $ \widetilde{X}/\!\!/G\rightarrow X/\!\!/G$ is obtained by blowing-up successively the (proper transforms of the) subvarieties $ Z_R^{ss}/\!\!/N(R) $.  Since the partial desingularization $ \widetilde{X}/\!\!/G $ has only finite quotient singularities, its intersection cohomology $ IH^*(\widetilde{X}/\!\!/G) $ with rational coefficients is isomorphic to the corresponding rational cohomology $ H^*(\widetilde{X}/\!\!/G) $, and so by the above results we know the Betti numbers of its intersection cohomology. Eventually, we will be able to find the intersection Betti numbers of $ X/\!\!/G $, by means of the following:

\begin{theorem}\cite[3.1]{Kir86}\label{thm:blowdown}
	In the above setting, the intersection Hilbert-Poincar\'{e} polynomial of the GIT quotient $ X/\!\!/G $ is related to that of the Kirwan blow-up via the equality
	\[ IP_t(X /\!\!/ G)=P_t(\widetilde{X}/\!\!/G)-\sum_{R\in\calR} B_R(t),\]
	where the error term is given by:
	\[ B_R(t)=\sum_{p+q=i}t^i \dim[H^p(\hat{Z}_{R}/\!\!/N^0(R))\otimes IH^{\hat{q}_R}(\PP \calN_x^R /\!\!/R)]^{\pi_0 N(R)}, \]
	where the integer $ \hat{q}_R=q-2 $ for $ q\leq \dim \PP\calN_x^R/\!\!/R $ and $ \hat{q}_R=q $ otherwise. The subvariety $ \hat{Z}_{R} $ is the strict transform of $ Z_R^{ss} $ in the appropriate stage of the resolution, while $ N(R)\subset G $ denotes the normaliser of $ R $. The GIT quotient $ \PP \calN_x^R /\!\!/R $ is constructed from the induced action of $ R $ on the normal slice $ \calN_x^R $ to the orbit $ G\cdot Z_R^{ss} $ in $ X^{ss} $ at a general point $ x\in Z_R^{ss} $.
\end{theorem}

\begin{remark}
	If $ \hat{Z}_{R}/\!\!/N^0(R) $ is simply connected, which is always the case in our situation, then the action of $ \pi_0 N(R) $ on the tensor product splits \cite[\S 2]{Kir86}, thus the error term for the subgroup $ R $ is
	\[ B_R(t)=\sum_{p+q=i}t^i\dim H^p(\hat{Z}_{R}/\!\!/N(R))\cdot \dim IH^{\hat{q}_R}(\PP \calN_x^R/\!\!/R)^{\pi_0 N(R)}.\]
\end{remark}

\subsection{Cohomology of blow-downs for $ (3, 3) $ curves in $ \PP^1 \times \PP^1 $} We want to apply Theorem \ref{thm:blowdown} to compute the intersection Betti numbers of the moduli space of non-hyperelliptic Petri-general curves of genus 4.
Now we will follow backwards the steps of the blow-down operations of \textit{A-curves}, then \textit{D-curves}, and eventually triple conics.

\subsubsection{A-curves} In the first step we need to blow-down the locus of \textit{A-curves}.
\begin{prop}
	For the group $ R_{A}\cong \CC^* $, we have 
	\begin{enumerate}[(i)]
		\item $ Z_{R_{A}}/\!\!/N(R_{A}) $ is a point.
		\item $ IP_t(\PP \calN_x^{R_{A}}/\!\!/R_{A})=1+2t^2+3t^4+4t^6+5t^8+4t^{10}+3t^{12}+2t^{14}+t^{16}. $
	\end{enumerate}
	The term $ B_{R_{A}}(t) $ is given by
	\begin{align*}
	B_{R_{A}}(t)&=t^2+t^4+2t^6+2t^8+2t^{10}+2t^{12}+t^{14}+t^{16}\\
	&\equiv A_{R_{A}}(t) \ \mod t^{10}.
	\end{align*}
	
\end{prop}
\begin{proof} For brevity we write $ R=R_{A} $, $ N=N(R_{A}) $ and $\PP^9 \cong \PP \calN_x^{R_{A}} $. (1) follows from the fact that $ N $ acts transitively on $ Z_{R}^{ss} $.
	
	In Lemma \ref{lem:weights2A5} the weights of the representation $\rho: R\rightarrow \Aut(\calN_x^R) $ were computed. It follows that there are no strictly-semistable points in $ \PP^9 $, so that the GIT quotient $ \PP^9/\!\!/R $ is a projective toric variety of dimension 8 with at worst finite quotient singularities. Thus $ IP_t(\PP^9/\!\!/R)=P_t(\PP^9/\!\!/R)=P_t^{R}((\PP^9)^{ss}) $ and using the usual $ R $-equivariantly perfect stratification (see Theorem \ref{thm:strata} and \ref{thm:equi}) we obtain:
	\begin{align*}
	P_t^{R}((\PP^9)^{ss})&=P_t(\PP^9)P_t(BR)-\sum_{0\neq \beta'\in \calB(\rho)}t^{2d(\beta')}P_t^R(S_{\beta'})\\
	&=\frac{1+...+t^{18}}{1-t^2}-2\frac{t^{10}+...+t^{18}}{1-t^2}\\
	&=1+2t^2+3t^4+4t^6+5t^8+4t^{10}+3t^{12}+2t^{14}+t^{16}.
	\end{align*}
	Now we need to know the dimensions $ \dim IH^{\hat{q}}(\PP^9/\!\!/R)^{\pi_0 N} $, where the action is induced by an action of $ \pi_0 N $ on $ \PP^9/\!\!/R $. We have seen that $ \pi_0 N \cong \mm $ acts on $ \PP^9/\!\!/R $ via permutation of the coordinates $ ((x_0:x_1), (y_0:y_1))\leftrightarrow ((x_1:x_0), (y_1:y_0)) $. Thus the action on the cohomology of $ \PP^9 $ is trivial, while $ \mm $ acts on the torus $ \CC^* $ via $ \lambda \leftrightarrow \lambda^{-1} $, hence in cohomology $ H^*(B\CC^*)=\QQ[c]$ by $ c\leftrightarrow -c $, and on the strata interchanging the positive-indexed ones with the negative-indexed ones. Eventually
	\begin{align*}
	IP_t(\PP^9/\!\!/R)^{\pi_0 N}&= \frac{1+...+t^{18}}{1-t^4}-\frac{t^{10}+...+t^{18}}{1-t^2}\\
	&= 1+t^2+2t^4+2t^6+3t^8+2t^{10}+2t^{12}+t^{14}+t^{16}.
	\end{align*}
	Now the final statement easily follows from the definition of $ B_R(t) $.
\end{proof}

\subsubsection{D-curves} In the second step, we need to blow-down the locus of \textit{D-curves}.
\begin{prop}
	For the group $ R_{D}\cong \CC^* $ with the notation as in the proof of Proposition \ref{prop:mainD}, we have 
	\begin{enumerate}[(i)]
		\item $ Z_{R_{D},1}/\!\!/N(R_{D}) $ is a simply connected surface and $ P_t(Z_{R_{D},1}/\!\!/N(R_{D}))=1+2t^2+t^4 $.
		\item $ IP_t(\PP \calN_x^{R_D}/\!\!/R_D)=1+2t^2+3t^4+4t^6+3t^8+2t^{10}+t^{12}. $
	\end{enumerate}
	The term $ B_{R_{D}}(t) $ is equal to
	\[ B_{R_D}(t)=t^2+3t^4+5t^6+7t^8+7t^{10}+5t^{12}+3t^{14}+t^{16}. \]
\end{prop}

\begin{proof} For brevity we write $ R=R_{D} $, $ N=N(R_{D}) $ and $ \PP^7\cong \PP \calN_x^{R_D} $. The GIT quotient $  Z_{R,1}/\!\!/N\cong Z_{R,1}/\!\!/(N/R) $ is a rational surface with finite quotient singularities, hence simply connected by \cite[Theorem 7.8]{Kol93}. Its cohomology can be computed by means of the equality \cite[1.17]{Kir86}:
	\[ H_{N}^*(Z_{R,1}^{ss})=(H^*(Z_{R,1}/\!\!/N^0)\otimes H^*(BR))^{\pi_0 N}. \]
	The action of $ \pi_0 N $ splits on the tensor product, because also $ Z_{R,1}/\!\!/N^0 $ is simply connected, giving:
	\[ H^*_N(Z_{R_{D},1}^{ss})=H^*(Z_{R_{D},1}^{ss}/\!\!/N)\otimes H^*(BR)^{\pi_0 N}.\]
	Recall that $ \pi_0 N=N/T=\mm \rtimes \mm =\mm \times \mm $: the first factor acts on $ R\cong \CC^* $ by inversion, while the second one acts trivially. Therefore
	\[ H^*(BR)^{\pi_0 N}=\QQ[c]^{\mm \times \mm}=\QQ[c^2], \ \deg(c)=2. \]
	In the proof of Proposition \ref{prop:mainD}, we have already computed $ P_t^N(Z_{R_{D},1}^{ss}) $, thus
	\[ P_t(Z_{R_{D},1}^{ss}/\!\!/N)=\frac{1+t^2}{1-t^2}(1-t^4)=1+2t^2+t^4, \]
	completing the proof of (i). 
	
	In Lemma \ref{lem:weightsD} the weights of the representation $\rho: R\rightarrow \Aut(\calN_x^R) $ were carried out. It follows that there are no strictly-semistable points in $ \PP^7 $, so that the GIT quotient $ \PP^7/\!\!/R $ is a projective variety of dimension 6 with at worst finite quotient singularities. Thus $ IP_t(\PP^7/\!\!/R)=P_t(\PP^7/\!\!/R)=P_t^{R}((\PP^7)^{ss}) $ and using the usual $ R $-equivariantly perfect stratification (see Theorem \ref{thm:strata} and \ref{thm:equi}) we obtain: 
	\begin{align*}
	P_t^{R}((\PP^7)^{ss})&=P_t(\PP^7)P_t(BR)-\sum_{0\neq \beta'\in \calB(\rho)}t^{2d(\beta')}P_t^R(S_{\beta'})\\
	&=\frac{1+...+t^{14}}{1-t^2}-2\frac{t^8+t^{10}(1+t^2)+t^{14}}{1-t^2}\\
	&=1+2t^2+3t^4+4t^6+3t^8+2t^{10}+t^{12}.
	\end{align*}
	Now we need to know the dimensions $ \dim IH^{\hat{q}}(\PP^7/\!\!/R)^{\pi_0 N} $. We have seen that $ \pi_0 N\cong \mm \times \mm$ acts on $ \PP^7/\!\!/R $ as follows: the first $ \mm $ factor via permutation of the coordinates $ ((x_0:x_1), (y_0:y_1))\leftrightarrow ((x_1:x_0), (y_1:y_0)) $, while the second one by interchanging the rulings of $ \PP^1 \times \PP^1 $. Thus the action on the cohomology of $ \PP^7 $ is trivial, while the first factor of $ \mm \times \mm $ acts on the torus $ \CC^* $ via $ \lambda \leftrightarrow \lambda^{-1} $, hence in cohomology $ H^*(B\CC^*)=\QQ[c] $ by $ c\leftrightarrow -c $, and the second factor does trivially. Moreover $ \pi_0 N $ acts on the strata interchanging the positive-indexed ones with the negative-indexed ones:
	\begin{align*}
	IP_t(\PP^7/\!\!/R)^{^{\pi_0 N}} &=\frac{1+...+t^{14}}{1-t^4}-\frac{t^8+...+t^{14}}{1-t^2} \\
	&= 1+t^2+2t^4+2t^6+2t^8+t^{10}+t^{12}.
	\end{align*}
	Now the final statement easily follows from the definition of $ B_R(t) $.
\end{proof}
\subsubsection{Triple conic}The last step is blowing-down the triple conics.
\begin{lemma}\label{lem:intPN3C}
	The intersection cohomology of the GIT quotient $ \PP \calN_x^{R_{C}}/\!\!/R_{C} $ is 
	\[ IP_t(\PP \calN_x^{R_{C}}/\!\!/R_{C})= 1+t^2+2t^4+2t^6+2t^8+2t^{10}+2t^{12}+t^{14}+t^{16}. \]
\end{lemma}
\begin{proof} For brevity we write $ R=R_{C}\cong \SL(2, \CC) $ and $ \PP^{11}=\PP \calN_x^{R_{C}} $. From the weights of the slice representation (Lemma \ref{lem:weights3C}) and the usual $ R $-equivariantly perfect stratification (see Theorem \ref{thm:strata} and \ref{thm:equi}) one can compute the equivariant Poincar\'{e} series of the semistable locus
	\[P_t^{R}((\PP^{11})^{ss})=\frac{1+...+t^{22}}{1-t^4}-\frac{t^{12}(1+t^2)+t^{16}(1+t^2)+t^{20}}{1-t^2}. \]
	Unfortunately, the space $ \PP^{11}/\!\!/R $ is not rationally smooth, thus $P_t^R((\PP^{11})^{ss})  $ is a priori neither $ P_t(\PP^{11}/\!\!/R) $ nor $ IP_t( \PP^{11}/\!\!/R) $. The remedy for this is first to blow-up the orbit associated to the subgroup $ T_1:=\{ \diag(t, t^{-1}): t\in \CC^* \}\subset R$, which fixes strictly polystable points. Using the same procedure as before, we obtain a partial desingularization $ \widetilde{\PP^{11}}/\!\!/R $, whose cohomology is related to the $ R $-equivariant cohomology of $ (\PP^{11})^{ss} $ by the error term (see Theorem \ref{thm:cohkirblow})
	\[ A_{T_1}(t)=\frac{1+t^2}{1-t^4}(t^2+...+t^{14})-\frac{1+t^2}{1-t^2}(t^8+t^{10}(1+t^2)+t^{14}). \]
	Hence the cohomology of the Kirwan blow-up is given by:
	\begin{align*}
	P_t(\widetilde{\PP^{11}}/\!\!/R)&= P_t^{R}((\PP^{11})^{ss})+A_{T_1}(t)\\
	&=1+2t^2+4t^4+5t^6+6t^8+5t^{10}+4t^{12}+2t^{14}+t^{16}.
	\end{align*}
	Now by the blowing-down procedure (see Theorem \ref{thm:blowdown}), we need to subtract the error term
	\[ B_{T_1}(t)=t^2+2t^4+3t^6+4t^8+3t^{10}+2t^{12}+t^{14}. \]
	Now the statement follows from $ IP_t( \PP^{11}/\!\!/R)=P_t(\widetilde{\PP^{11}}/\!\!/R)-B_{T_1}(t) $.
\end{proof}
\begin{prop}
	For the group $ R_{C}\cong \SL(2, \CC) $, the error term $ B_{R_{C}}(t) $ is given by
	\begin{align*}
	B_{R_{C}}(t)&= t^2+t^4+2t^6+2t^8+2t^{10}+2t^{12}+t^{14}+t^{16}\\
	&\equiv A_{R_{C}}(t) \ \mod t^{10}.
	\end{align*}
	
\end{prop}
\begin{proof} The result easily follows from the definition of $ B_{R_{C}}(t) $, after noticing that $ Z_{R_{C}}/\!\!/N(R_{C}) $ is a point and the group $ \pi_0 N(R_{C}) $ acts trivially on $ IH^*(\PP \calN_x^{R_{C}}/\!\!/R) $ (cf. Proposition \ref{prop:main3C}), which we computed in Lemma \ref{lem:intPN3C}.
\end{proof}
\subsection{Intersection cohomology of $ M $} We complete the proof of Theorem \ref{thm:intM}.

\begin{proof}[Proof of Theorem \ref{thm:intM}]
From Theorem \ref{thm:blowdown} putting all the previous results together, we obtain that the intersection Hilbert-Poincar\'{e} polynomial of the moduli space of non-hyperelliptic Petri-general genus 4 curves $ M=X/\!\!/G $ is 
\begin{align*}
IP_t(M)&=P_t(\widetilde{M})-\sum_{R\in \calR} B_R(t)\\
&=P_t^G(X^{ss})+\sum_{R\in \calR}(A_R(t)-B_R(t))\\
&\equiv 1+t^2+2t^4+2t^6+4t^8+0-t^8+0 \ \mod t^{10}\\
&\equiv 1+t^2+2t^4+2t^6+3t^8 \ \mod t^{10}.
\end{align*}
\end{proof}

Together with Theorem \ref{thm:cohoblow}, this also completes the proof of the main Theorem \ref{thm:main}.

\begin{remark}\label{rmk:coho}(cf. \cite[3.4]{Kir86})
	As a by-product of our result, we are able to determine the ordinary Betti numbers
	\[ H^i(X/\!\!/G)=IH^i(X/\!\!/G) \ \mathrm{for} \ 12 \leq i\leq 18  \] 
	and
	\[ H^i(X^{s}/G)=IH^i(X/\!\!/G) \ \mathrm{for} \ 0 \leq i\leq 6  \] 
	where $ X^{s}/G=X/\!\!/G\smallsetminus \bigcup_{R\in \calR} Z_R/\!\!/N(R) $ is the orbit space of GIT-stable curves.
\end{remark}
\subsection{Geometric interpretation} In conclusion, we give a geometric interpretation of some Betti numbers of the compactification $ M $, by describing the classes of curves generating the cohomology spaces.

Let $ U\subset M $ be the affine open subset corresponding to smooth non-hyperelliptic Petri-general curves of genus four. Tommasi (\cite[Theorem 1.2]{Tom05}) computed the rational cohomology of $ U $, as geometric quotient of the complement of a discriminant, namely
\[ H^i(U)=\begin{cases}
1 \quad i=0, 5\\
0 \quad \mathrm{otherwise.}
\end{cases} \]
We now consider the Gysin long exact sequence (cf. \cite[\S 19.1 (6)]{Ful98}) associated to the inclusion $ U \hookrightarrow M $:
\begin{equation}\label{gysin}
... \rightarrow H_{k+1}(U)\rightarrow H_{k}(M \smallsetminus U) \rightarrow H_k(M) \rightarrow H_k(U) \rightarrow ...
\end{equation}
where $ H_* $ denotes the rational Borel-Moore homology theory (cf. \cite[Example 19.1.1]{Ful98}).
As $ U $ has at most finite quotient singularities, by Poincar\'{e} duality $ \dim H_{k+1}(U)=1 $ for $ k=12, 17 $ and vanishes in all other degrees. 

The dimensions of $ H_{k}(M\smallsetminus U)\cong H_{k}(X^s/G \smallsetminus U) $, for $ k \geq 12 $, can be also computed from Remark \ref{rmk:coho} via the Gysin sequence related to the inclusion $ U \hookrightarrow X^s/G $. 
Therefore, the geometry of the curves in $ M\smallsetminus U $ suggests the following geometric interpretation of the Betti numbers:
\begin{itemize}
	\item $ H_{18}(M) $ is obviously generated by the fundamental class of $ M $;
	\item $ H_{16}(M) $ is generated by the fundamental class of $ M\smallsetminus U $, i.e. the locus of singular curves;
	\item $ H_{14}(M) $ is generated by the fundamental classes of the following subvarieties of $ M\smallsetminus U $: the closure of the locus of curves with at least two nodes and the closure of the locus of curves with a cusp;
	\item $ H_{12}(M) $ is generated by the fundamental classes of the following subvarieties of $ M\smallsetminus U $: the closure of the locus of curves with at least three points in general position, the locus of reducible curves with a line as component and the closure of the locus of curves with at least a node and cusp. These three classes generate $ H_{12}(M\smallsetminus U)\cong \QQ^3 $, but are linearly dependent in $ H_{12}(M)\cong \QQ^2 $ and the space of relations con be identified with $ H_{13}(U)\cong \QQ $. 
\end{itemize}

Similar (but dual) considerations can be applied to the Betti numbers of the stable quotient $ X^s/G $. The geometric interpretation explained above hence confirms the results about $ IH^i(M) $ for $ i\leq 6 $.

\bibliographystyle{alpha}
\bibliography{References}

\begin{thebibliography}{CMGHL}

\bibitem[AB83]{AB83}
M.~F. Atiyah and R.~Bott.
\newblock The {Y}ang-{M}ills equations over {R}iemann surfaces.
\newblock {\em Philos. Trans. Roy. Soc. London Ser. A}, 308(1505):523--615,
  1983.

\bibitem[BBD82]{BBD82}
A.~A. Be\u{\i}linson, J.~Bernstein, and P.~Deligne.
\newblock Faisceaux pervers.
\newblock In {\em Analysis and topology on singular spaces, {I} ({L}uminy,
  1981)}, volume 100 of {\em Ast\'{e}risque}, pages 5--171. Soc. Math. France,
  Paris, 1982.

\bibitem[BT07]{BT07}
J.~Bergstr\"{o}m and O.~Tommasi.
\newblock The rational cohomology of {$\overline{M_4}$}.
\newblock {\em Math. Ann.}, 338(1):207--239, 2007.

\bibitem[CMGHL]{CMGHL}
S.~Casalaina-Martin, S.~Grushevsky, K.~Hulek, and R.~Laza.
\newblock Cohomology of the moduli space of cubic threefolds.
\newblock arXiv:1904.08728.

\bibitem[CMJL12]{CMJL12}
S.~Casalaina-Martin, D.~Jensen, and R.~Laza.
\newblock The geometry of the ball quotient model of the moduli space of genus
  four curves.
\newblock In {\em Compact moduli spaces and vector bundles}, volume 564 of {\em
  Contemp. Math.}, pages 107--136. Amer. Math. Soc., Providence, RI, 2012.

\bibitem[CMJL14]{CMJL14}
S.~Casalaina-Martin, D.~Jensen, and R.~Laza.
\newblock Log canonical models and variation of {GIT} for genus 4 canonical
  curves.
\newblock {\em J. Algebraic Geom.}, 23(4):727--764, 2014.

\bibitem[Fed12]{Fed12}
M.~Fedorchuk.
\newblock The final log canonical model of the moduli space of stable curves of
  genus 4.
\newblock {\em Int. Math. Res. Not. IMRN}, (24):5650--5672, 2012.

\bibitem[FP15]{FP15}
C.~Faber and R.~Pandharipande.
\newblock {Tautological and non-tautological cohomology of the moduli space of
  curves.}
\newblock In {\em {Handbook of moduli. Volume I}}, pages 293--330. Somerville,
  MA: International Press; Beijing: Higher Education Press, 2015.

\bibitem[Ful98]{Ful98}
W.~Fulton.
\newblock {\em Intersection theory}, volume~2 of {\em Results in Mathematics
  and Related Areas.}
\newblock Springer-Verlag, Berlin, second edition, 1998.

\bibitem[GH78]{GH78}
P.~Griffiths and J.~Harris.
\newblock {\em Principles of algebraic geometry}.
\newblock Wiley-Interscience [John Wiley \& Sons], New York, 1978.
\newblock Pure and Applied Mathematics.

\bibitem[Has05]{Has05}
B.~Hassett.
\newblock {Classical and minimal models of the moduli space of curves of genus
  two.}
\newblock In {\em {Geometric methods in algebra and number theory}}, pages
  169--192. 2005.

\bibitem[HH09]{HH09}
B.~Hassett and D.~Hyeon.
\newblock {Log canonical models for the moduli space of curves: the first
  divisorial contraction.}
\newblock {\em {Trans. Am. Math. Soc.}}, 361(8):4471--4489, 2009.

\bibitem[HH13]{HH13}
B.~Hassett and D.~Hyeon.
\newblock {Log minimal model program for the moduli space of stable curves: the
  first flip.}
\newblock {\em {Ann. Math. (2)}}, 177(3):911--968, 2013.

\bibitem[HL14]{HL14}
D.~Hyeon and Y.~Lee.
\newblock {A birational contraction of genus 2 tails in the moduli space of
  genus 4 curves I.}
\newblock {\em {Int. Math. Res. Not.}}, 2014(13):3735--3757, 2014.

\bibitem[Kir84]{Kir84}
F.~Kirwan.
\newblock Cohomology of quotients in symplectic and algebraic geometry.
\newblock 31:i+211, 1984.

\bibitem[Kir85]{Kir85}
F.~Kirwan.
\newblock Partial desingularisations of quotients of nonsingular varieties and
  their {B}etti numbers.
\newblock {\em Ann. of Math. (2)}, 122(1):41--85, 1985.

\bibitem[Kir86]{Kir86}
F.~Kirwan.
\newblock Rational intersection cohomology of quotient varieties.
\newblock {\em Invent. Math.}, 86(3):471--505, 1986.

\bibitem[Kir89]{Kir89}
F.~Kirwan.
\newblock Moduli spaces of degree {$d$} hypersurfaces in {${\bf P}_n$}.
\newblock {\em Duke Math. J.}, 58(1):39--78, 1989.

\bibitem[KL89]{Kir88}
F.~Kirwan and R.~Lee.
\newblock The cohomology of moduli spaces of {$K3$} surfaces of degree {$2$}.
  {I}.
\newblock {\em Topology}, 28(4):495--516, 1989.

\bibitem[Kol93]{Kol93}
J.~Koll\'{a}r.
\newblock Shafarevich maps and plurigenera of algebraic varieties.
\newblock {\em Invent. Math.}, 113(1):177--215, 1993.

\bibitem[KW06]{KW06}
F.~Kirwan and J.~Woolf.
\newblock {\em An introduction to intersection homology theory}.
\newblock Chapman \& Hall/CRC, Boca Raton, FL, second edition, 2006.

\bibitem[MFK94]{MFK94}
D.~Mumford, J.~Fogarty, and F.~Kirwan.
\newblock {\em Geometric invariant theory}, volume~34 of {\em Results in
  Mathematics and Related Areas (2)}.
\newblock Springer-Verlag, Berlin, third edition, 1994.

\bibitem[Mum]{Mum83}
D.~Mumford.
\newblock {Towards an enumerative geometry of the moduli space of curves.}
\newblock {Arithmetic and geometry, Pap. dedic. I. R. Shafarevich, Vol. II:
  Geometry, Prog. Math. 36, 271-328 (1983).}

\bibitem[Tom05]{Tom05}
O.~Tommasi.
\newblock Rational cohomology of the moduli space of genus 4 curves.
\newblock {\em Compos. Math.}, 141(2):359--384, 2005.

\end{thebibliography}

\end{document}